\documentclass[11pt,twoside,reqno]{amsart}
\usepackage{amsmath, amsfonts, amsthm, amssymb, graphicx, cite}
\usepackage{esint}
\usepackage{mathrsfs}
\usepackage{tikz}
\usetikzlibrary{matrix,arrows,decorations.pathmorphing}
\usepackage{setspace}
\usepackage{dsfont}
\usepackage{epigraph}

\theoremstyle{plain}
\newtheorem{thm}{Theorem}[section]
\newtheorem{lma}{Lemma}[section]

\theoremstyle{definition}
\newtheorem{defin}{Definition}[section]
\theoremstyle{remark}
\newtheorem{rem}{Remark}[section]

\newcommand{\ep}{\varepsilon}
\newcommand{\eps}{\epsilon}
\newcommand{\LL}{\langle}
\newcommand{\RR}{\rangle}
\newcommand{\pd}{\partial}

\newcommand{\Ex}{\mathbb{E}}

\newcommand{\vp}{\varphi}

\newcommand{\nb}{\mathrm{nb}}

\newcommand{\A}{\mathbf{A}}
\newcommand{\B}{\mathbf{B}}

\newcommand{\ab}{{\boldsymbol\alpha}}

\newcommand{\mX}{\mathfrak{X}}
\newcommand{\se}{\mathcal{S}}
\newcommand{\1}{\mathbf{1}}
\newcommand{\mY}{\mathfrak{Y}}
\renewcommand{\d}{{\rm d}}

\newcommand{\R}{\mathbb{R}}
\newcommand{\T}{\mathbb{T}}

\numberwithin{equation}{section}

\title[Invariant Measures for Nonlinear Stochastic Balance Laws]
{Invariant Measures for Nonlinear Conservation Laws Driven by Stochastic Forcing}

\author
{Gui-Qiang G. Chen \qquad Peter H.C. Pang}
\address{Gui-Qiang G. Chen,
Mathematical Institute, University of Oxford,
Oxford, OX2 6GG, UK}
\email{chengq@maths.ox.ac.uk}
\address{Peter Pang,
Mathematical Institute, University of Oxford,
Oxford, OX2 6GG, UK}
\email{ho.pang@jesus.ox.ac.uk}

\begin{document}
\begin{abstract}
Some recent developments in the analysis of long-time behaviors of stochastic solutions of nonlinear conservation laws
driven by stochastic forcing are surveyed.
The existence and uniqueness of invariant measures are established for anisotropic degenerate
parabolic-hyperbolic conservation laws of second-order driven by white noises.
Some further developments, problems, and challenges in this direction are also discussed.
\end{abstract}
\maketitle

\section{Introduction}
The analysis of long-time behaviors of global solutions is the second pillar
in the theory of partial differential equations (PDEs),
after the analysis of well-posedness.
For the analysis of solution behaviors in the asymptotic regime,
we seek to understand the global properties of the solution map, such as attracting or repelling sets,
stable and unstable fixed points, limiting cycles, or chaotic behaviors that are properly determined
by the entire system rather than a given path.

The introduction of noises usually serves to model dynamics phenomenologically -- dynamics too complicated to model from first principles,
or dynamics only the statistics of which are accurately known, or dynamics almost inherently random such as the decision of
many conscious agents -- or a combination of such behaviors.
Mathematically, noises introduce behaviors that differ from deterministic dynamics,
displaying much richer phenomena such as effects of dissipation,
ergodicity, among others ({\it cf.} \cite{Ver1981,FGP2010,DFV2014,FFPV2017,KR2005,DFPR2013}
and the references cited therein).
These phenomena are  of intrinsic interest.

In this paper, we focus our analysis mainly on white-in-time noises.
Indeed, they are the most commonly studied class of noises, though space-time white noises (such as in \cite{DD2002})
and more general rough fluxes ({\it e.g.} \cite{LPS2013,LPS2014,PZ2007}) have also been considered.
The reason for the prevalence of white noises as a basic model is not difficult to understand.
First, Brownian motion occupies the unusual position of being simultaneously a martingale and a L\'evy process.
More importantly,  with increments that are not only independent but also normally distributed,
it commands a level of universality by virtue of the central limit theorem.
Some of the ideas, techniques, and approaches presented here can be applied to equations with more general
or other forms of stochastic forcing.

The organization of this paper is as follows:
In \S 2,
the notion of invariant measures is first introduced,
then the Krylov-Bogoliubov approach
for the existence of invariant measures is presented,
and some methods for the uniqueness of invariant measures
including the strong Feller property and the coupling method are discussed.
In \S3, some recent developments in the analysis of long-time behaviors of solutions
of nonlinear stochastic PDEs are discussed.
In \S4, we establish the existence of invariant measures for nonlinear anisotropic parabolic-hyperbolic equations driven by white noises.
In \S5, we establish the uniqueness of invariant measures  for the stochastic anisotropic parabolic-hyperbolic equations.
In \S6, we present some further developments,  problems, and challenges
in this research direction.

\section{Invariant Measures}

In this section,  we first introduce the notion of invariant measures for random dynamic systems,
and then present several approaches to establish the
existence and uniqueness of invariant measures.

\subsection{Notion of invariant measures}

The notion of invariant measures on a dynamical system is quite straightforward.
Let $(\mX,\Sigma,\mu)$ be a measure space, and let $S: \mX \to \mX$ be a map.
System $(\mX,\Sigma,\mu,S)$ is a measure-preserving system
if $\mu(S^{-1}A) = \mu(A)$ for any $A \in \Sigma$.
Then $\mu$ is called an invariant measure of map $S$.

On a random dynamic system (RDS), there is an added layer of complexity.
We follow the standard definitions in \cite{Arn1998};
see also \cite{Cra2002} for further references on RDSs and \cite{Fla1991,Fla1995}
in a specifically parabolic SPDE context.

Let $(\Omega, \mathcal{F}, \mathbb{P})$ be a probability space,
and let $\theta_t: \Omega \to \Omega$ be a collection of probability-preserving maps.
A {\it measurable RDS} on a measurable space $(\mX,\Sigma)$ over
quadruple $(\Omega, \mathcal{F}, \mathbb{P}, \theta_t)$ is a map:
\[
\varphi\,:\, \mathbb{R} \times \Omega \times \mX \to \mX
\]
satisfying the following:
\begin{itemize}
\item[(i)] Measurability: $\,\varphi$ is $(\mathcal{B}(\mathbb{R})\otimes \mathcal{F}\otimes \Sigma, \Sigma)$--measurable{\rm ;}
\item[(ii)] Cocycle property: $\,\varphi(t,\omega) = \varphi(t,\omega,\cdot): \mX \to \mX$ is a cocyle over $\theta$:
\begin{align*}
&\varphi(0,\omega) = \mathrm{id}_\mX,\\
&\varphi(t + s,\omega) =  \varphi(s,\theta_t \omega) \varphi(t,\omega),
\end{align*}
where $\mathcal{B}(\mathbb{R})$ denotes the collection of Borel sets in $\mathbb{R}$.
\end{itemize}

We think of $\Omega \times \mX \to \Omega$ as a fibre bundle with fibre $\mX$.
On the bundle, we have the {\it skew product} defined as $\Theta_t = (\theta_t,\varphi)$.
Then the invariant measures can be defined as follows:

\begin{defin}[Invariant measures]
An {\it invariant measure} on a RDS $\varphi$ over $\theta_t$ is a probability measure $\mu$
on $(\Omega \times \mX, \mathcal{F}\otimes \Sigma)$ satisfying
\[
(\Theta_t)_*\mu = \mu,\quad\; \mu(\cdot,\mX) = \mathbb{P},
\]
where $(\Theta_t)_*\mu:=\mu\circ (\Theta_t)^{-1}$ is the pushforward measure.
\end{defin}

Any probability measure $\mu$ on $\Omega \times \mX$ admits a disintegration:
\[
\mu(\omega, u) = \nu_\omega(u) \mathbb{P}(\omega).
\]
A measure $\nu_\omega$ is {\it stationary} if
\[
\varphi(t,\omega)_*\nu_\omega = \nu_{\theta_t\omega}.
\]
Let $\{\mathscr{F}_t\}_{t \ge 0}$ be a filtration associated with the RDS $\varphi$, {\it i.e.} an increasing
sequence of $\sigma$--sub-algebras of $\mathscr{F}$ by which $\varphi(t,\cdot,x)$ is measurable (adapted).
A {\it Markov} invariant measure is an invariant measure for which
map: $\omega \mapsto \nu_\omega(\Gamma)$ is $(\mathscr{F}_0,\mathcal{B}(\mathbb{R}))$--measurable
for any $\Gamma \in \mathcal{B}(\mX)$ \cite[\S 4.2.1]{KS2012}.

The disintegration of measures is unique. There is a one-to-one correspondence
between a Markov invariant measure and a stationary measure  \cite{Cra1991,Led1986,LeJ1987}.
Associated with an invariant measure is a random attracting set,
which is generalized from the deterministic context ({\it cf.} \cite{CF1994,CDF1997}).
In the context of dissipative PDEs perturbed by noises, it can be shown that the Hausdorff dimension
of an attracting set is finite by the methods similar to those used in the deterministic
case (see \cite{Tem1997} and the references cited therein) of linearizing the flow and estimating the sums
of global Lyapunov exponents ({\it cf}. \cite{CF1994,Sch1997,Deb1997,Deb1998}).

\subsection{Approaches for the existence of invariant measures}\label{sec:krylov_bogoliubov}

There are several approaches to establish the existence of invariant measures.
One of the approaches is the Krylov-Bogoliubov approach, as we are going to discuss here.
Another approach is via Khasminskii's theorem \cite{DZ2008}.
Both of them are based on the compactness property
provided by the Prohorov theorem.

We first recall that a sequence of probability measures $\{\nu_n\}$ on a measure space $\mX$ is \emph{tight}
if, for every $\eps > 0$, there is a compact set $K_\eps \subseteq \mX$ such that
\[
\nu_n(\mX\backslash K_\eps) \leq \eps  \qquad\,\,\, \mbox{uniformly in $n$}.
\]

\begin{lma}[Prohorov theorem]
A tight sequence of probability measures $\nu_n$ is weak*--compact in the space of \emph{probability} measures{\rm ;}
that is, there exist a subsequence {\rm (}still denoted{\rm )}
$\nu_n$ and  a probability measure $\nu$
such that $\nu_n \overset{*}{\rightharpoonup} \nu$.
\end{lma}

\begin{thm}[Krylov-Bogoliubov theorem]\label{thm:KrylovBogoliubov}
Let $\mathscr{P}_t$ be a semigroup satisfying the {\emph Feller property} that $\phi \in C(\mX)$ implies $\mathscr{P}_s\phi \in C(\mX)$ for any $s > 0$,
and let $\mu$ be a probability measure on $\mX$ such that the measure sequence{\it :}
\begin{equation}\label{2.1-a}
\nu_T = \frac{1}{T} \int_0^T \mathscr{P}_t^*\mu\;\d t
\end{equation}
is tight.
Then there exists an invariant measure for $\mathscr{P}_t$.
\end{thm}

The key of its proof is
that the invariant measure generated by the Krylov-Bogoliubov theorem
is the weak*--limit of $\nu_T$ as $T\to \infty$.
This can be seen as follows:
By the Prohorov theorem, the tight sequence has a weakly converging subsequence (still denoted as) $\{\nu_T\}$
for $T$ ranging over a unbounded subset of $\mathbb{R}$, whose limit is $\nu_*$. Then
\begin{align*}
\LL \mathscr{P}^*_s \nu_*, \vp\RR_{\mX}
 = &\, \LL  \nu_*,  \mathscr{P}_s\vp \RR_\mX\\
 = &\,\lim_{T \to \infty} \frac{1}{T}\int_0^T  \LL  \mathscr{P}^*_t \mu, \mathscr{P}_s\vp\RR_\mX \;\d t \\
= &\, \lim_{T \to \infty} \frac{1}{T}\int_0^T  \LL \mathscr{P}^*_{t+s} \mu, \vp\RR_\mX \;\d t \\
= &\, \lim_{T \to \infty} \frac{1}{T}\int_0^T  \LL  \mathscr{P}^*_{t} \mu, \vp\RR_\mX \;\d t
+ \lim_{T \to \infty} \frac{1}{T}\int_T^{T + s}  \LL \mathscr{P}^*_{t} \mu, \vp\RR_\mX \;\d t
\\
&\, - \lim_{T \to \infty} \frac{1}{T}\int_0^{s} \LL \mathscr{P}^*_{t} \mu, \vp\RR_\mX \;\d t \\
=&\,\lim_{T\to \infty}\int_\mX \vp(u)d\nu_T.
\end{align*}
In the above, we require the Feller property to execute the first equality, as $\mathscr{P}_s \vp$ has to remain continuous.
With this, the second and third terms after the fourth equality above
tend to zero in the limit $T \to \infty$, as $s$ is fixed.

The following lemma provides two sufficient conditions for the tightness of $\{\nu_T\}$.

\begin{lma} A measure sequence $\{\nu_T\}$ is tight if one of the following conditions holds{\rm :}
\begin{enumerate}
\item[(i)] $\{\mathscr{P}_t^*\mu\}$  is tight{\rm ;}

\smallskip
\item[(ii)] $\{\mathscr{P}_t\}$ are compact for $t > 0$ so that
$$
\mathscr{P}_t(\mX) \subseteq \mY \qquad \,\,\mbox{for almost all $t>0$}
$$
for a Banach space $\mY$ such that there is a compact embedding $\mY \hookrightarrow \mX$
and there exists $C>0$ independent of $T$ so that
\begin{align}\label{eq:bbd_avg}
\frac{1}{T}\int_0^T \|\mathscr{P}_tu_0\|_{\mY}\;\d t  \le  C
\qquad\mbox{for any $u_0\in \mX$};
\end{align}
In addition, $\mu=\delta_{u_0}$ for some $u_0\in \mX$.
\end{enumerate}
\end{lma}

\begin{proof}
For (i), we know that, for any $\eps>0$, there is a compact set $K_\eps\subseteq \mX$ such that
$$
\mathscr{P}_t^*\mu(\mX\setminus K_\eps)\le \eps  \qquad \mbox{uniform in $t>0$}.
$$
Then
$$
\nu_T(\mX\setminus K_\eps)\le \frac{1}{T}\int_0^T
\mathscr{P}_t^*\mu(\mX\setminus K_\eps)\, dt\le \eps.
$$

\medskip
For (ii),
let $K_R = \{ u\in \mX : \|u\|_\mY \leq R\}$.
Since $\mY \hookrightarrow \mX$ is compact, $K_R$ is compact in $\mX$.

\medskip
If $u \in \mX \setminus K_R$, then $\|u\|_\mY > R$.
Writing $f(\cdot) = \|\cdot\|_\mY$,
then
\begin{align*}
\nu_T(\mX\setminus K_R) \leq  &\, \nu_T(\{f(u) > R\}).
\end{align*}
Applying the Markov inequality to $f$, we have
\begin{align*}
\nu_T(\mX\setminus K_R) \leq &\, \frac{1}{R} \int_\mX f(u)\, \d \nu_T(u)\\
= &\, \frac{1}{RT} \int_0^T \int_\mX (\mathscr{P}_tf)(u)\, \d\mu(u) \;\d t .
\end{align*}

Since $\mu = \delta_{u_0}$ for some  $u_0\in \mX$, then
\[
\frac{1}{RT} \int_0^T \int_\mX (\mathscr{P}_tf)(u)\, \d\mu(u)\;\d t
=  \frac{1}{RT} \int_0^T (\mathscr{P}_tf)(u_0)  \;\d t
= \frac{1}{RT} \int_0^T f(u(t))\;\d t .
\]

Therefore, if the temporal average \eqref{eq:bbd_avg} is bounded,
then, for any $\eps>0$, we can choose $R>\frac{1}{\eps}$ to conclude
$$
\nu_T(\mX\setminus K_R) \leq  \, \nu_T(\{f(u) > R\})<\eps.
$$
In this way, a compact set $K_R$ has been found such that $\nu_T(\mX\setminus K_R) \leq \eps$,
which implies that $\{\nu_T\}$ is tight.
\end{proof}

\medskip
This framework can be further refined.
An example of such an extension can be found in \cite{CGT2018}, in which the Feller property could not be proved
in the context of the one-dimensional stochastic Navier-Stokes equations.
Whilst the Feller condition is not available, the continuous dependence (without rates) can be shown.
By using the continuous dependence,
a class of functions, $\mathcal{G} \supseteq C(\mX)$, is defined
so that $\mathcal{G}$ is continuous on the elements of the solution space with finite energy,
though not necessarily the entire solution space.
With these, it has been shown in  \cite{CGT2018}
that $\mathscr{P}_t$ is invariant under $\mathcal{G}$.
Then the existence of invariant measures is proved in two steps:
First, an energy bound is employed to yield the tightness, so that the existence of a limiting measure is shown to exist;
then
the limiting measure is shown to be invariant (without invoking the Feller property)
by using the continuity condition imposed on $\mathcal{G}$ and
following the arguments as in the proof of Theorem \ref{thm:KrylovBogoliubov}.

\subsection{Approaches for the proof of the uniqueness of invariant measures}\label{sec:im_uniqueness}

It is well known that the invariant measures of a map form a convex set
in the probability space on $X$.
By the Krein-Milman theorem, a convex set is the closure of convex combinations
of its extreme points.
These extreme points $\mu$ happen to be ergodic measures,
which are characterized as the property that, for a measurable subset $A \subseteq \mX$,
\[
\mu((\mathcal{S}^{-1}A) \Delta A) = 0\,\,\, \Longleftrightarrow\, \,\,\mu(A) = 0 \mbox{ or } \mu(A) = 1,
\]
where $A\Delta B:=(A\setminus B)\cup (B\setminus A)$.

Ergodic measures heuristically carve up the solution space into essentially disjoint subsets,
since any two ergodic measures of a process either coincide or are singular with respect to one another.
This is a simple consequence of the property stated above.

It also follows from the extremal property of ergodic measures that, if there are more than one invariant measure,
then there are at least two ergodic measures.

\medskip
There are several approaches to establish the uniqueness of invariant measures.

\smallskip
{\bf The Strong Feller Property}: This is one of the common conditions used
to ensure the uniqueness.

\begin{defin}[Strong Feller property]\label{defin:strongfeller}
A Markov transition semigroup $\mathscr{P}_t$ is of the strong Feller property at time $t$
if $\mathscr{P}_t \varphi$ is continuous for every bounded
measurable $\varphi: \mX \to \mathbb{R}$.
\end{defin}

The strong Feller property guarantees the uniqueness of invariant
measures \cite{Doo1948,Kha1960}; see also \cite[Theorem 5.2.1]{DZ2008},
\cite{MS1998}, and the references cited therein.

The strong Feller property always holds for transition semigroups of processes
associated with nonlinear stochastic evolution equations
with Lipschitz nonlinear coefficients and nondegenerate diffusion ({\it e.g.} \cite{PZ1995}).

\medskip
{\bf The Coupling Method}: This method is a powerful tool in probability theory
introduced in  Doeblin-Fortet\cite{DF1937,Doe1940},
which can be used to show the uniqueness of invariant measures.

The general argument proceeds as follows:
Let $X_t$ be a Markov process with initial distribution $\mu_0$,
and let $Y_t$ be an independent copy of the process with an initial distribution
that is an invariant measure $\mu$.
Then the first meeting time $\mathcal{T}$ is a stopping time,
and the process defined by
\begin{align*}
Z_t = \begin{cases}
\, X_t \quad &\mbox{for $t < \mathcal{T}$},\\
\, Y_t \quad &\mbox{for $t \geq \mathcal{T}$}
\end{cases}
\end{align*}
is also a copy of $X_t$ by the strong Markov property.

Using the definition of $Z_t$, we can write
\begin{align*}
\mathscr{P}_t^* \mu_0 - \mu
=&\, (Z_t)_* \mathbb{P} - (Y_t)_*\mathbb{P}\\
=&\, (\1_{\{t <\mathcal{T}\}}Z_t)_*\mathbb{P}+(\1_{\{t \geq \mathcal{T}\}}Z_t)_*\mathbb{P}
  -(\1_{\{t < \mathcal{T}\}}Y_t)_*\mathbb{P} - (\1_{\{t \geq \mathcal{T}\}}Y_t)_*\mathbb{P}\\
=&\, (\1_{\{t < \mathcal{T}\}}Z_t)_* \mathbb{P} - (\1_{\{t < \mathcal{T}\}}Y_t)_*\mathbb{P}.
\end{align*}
Then the total variation norm of $\mathscr{P}_t^* \mu_0 - \mu$ can be estimates as
\begin{align*}
\left\|\mathscr{P}_t^* \mu_0 - \mu \right\|_{TV}
\leq  &\, \int (\1_{\{t < \mathcal{T}\}}Z_t)_* \d\mathbb{P}(u) +  \int (\1_{\{t < \mathcal{T}\}}Y_t)_* \d\mathbb{P}(u)\\
 = &\, \mathbb{P}(\{t < \mathcal{T}\}).
\end{align*}

Assume that $\mathcal{T}$ can be shown to be almost surely finite.
Then, as $t \to \infty$, we see that $\mathscr{P}_t^* \mu_0 \to \mu$, and there is only one invariant measure.

The coupling method has other applications in various different settings and can be implemented in qualitatively different ways;
see also \cite{Lin1992,Vil2009} and the references cited therein.

\smallskip
In our applications for the uniqueness of invariant measures for stochastic anisotropic parabolic-hyperbolic equations in \S 5,
$\mathcal{T}$ will be slightly modified to be the time of entry into a small ball.
Moreover, instead of the use of independent copies, we take two solutions starting at different initial data,
since our Markov processes are solutions of the equations with pathwise uniqueness properties.

First, we show in \S \ref{sec:uniquenessI} that the two solutions $u$ and $v$ enter a given ball in finite time, almost surely.
This is a stopping time.
From this, by the strong Markov property, we construct a sequence of increasing, almost surely finite stopping
times in (\ref{eq:stopping_time}),
which are spaced at least $T$ apart, for some $T>0$ later to be fixed.

Then we show in \S \ref{sec:uniquenessII} that, for a well-chosen $T >0$,
if a solution starts within the same given ball,
and the noise is uniformly small in $W^{1,\infty}_x$ over a duration of length $T$, then
the temporal average of $\|u(t)\|_{L^1_x}$ over that temporal interval can be taken to be smaller than some $\eps$.
Since the noise is $\sigma(x) W$, the uniform smallness in $W^{1,\infty}_x$ over an interval $[\mathcal{T}, \mathcal{T} + T]$
depends entirely on the size of $W$.

We see that, for $T >0$, the probability that the change in the noise remains small between $\mathcal{T}$ and $\mathcal{T} + T$
is strictly positive.
By the strong Markov property, we can replace $\mathcal{T}$ with any other stopping time
({\it e.g.} the one in the sequence constructed) spaced at least $T$ apart.
Using the $L^1$--contraction,
we show finally in \S \ref{sec:uniquenessIII} that the probability that the difference between the two solutions
remains large for all intervals $[\mathcal{T}, \mathcal{T} + T]$,
with $\mathcal{T}$ in the sequence of increasing stopping times, is bounded by the probability that
the noise is large in $W^{1,\infty}$ over all such sequences.
This must be vanishingly small, as the probability is strictly less than one on each individual sequence.

\section{Nonlinear Conservation Laws driven by Stochastic Forcing}

In this section, we discuss one strand of the recent developments in the analysis of long-time behaviors of
global solutions of nonlinear conservation laws driven by stochastic forcing.

\subsection{The stochastic Burgers equation}
The Burgers equation is the archetypal nonlinear transport equation in many ways.
The stochastic Burgers equation has also been used in turbulence and interface dynamics modelling;
see \cite{DDT1994,GN1999,LNP2000,HV2011} and the references cited therein.

The existence of a non-trivial invariant measure of the process associated
to the one-dimensional Burgers equation driven by an additive spatially
periodic white noise was first derived in Sinai \cite{Sin1991}.

The long-time behavior of global solutions of the Burgers equation in one spatial dimension driven by space-time white noise
has also been considered in the form:
\[
\pd_t u + \pd_x \big(\frac{u^2}{2}\big)=\pd_{xx}^2 u + \pd^2_{xt} \tilde{W},
\]
where $\tilde{W}:=\tilde{W}(x,t)$
is a zero-mean Gaussian process with a covariance function given by
\[
\Ex[\tilde{W}(x,t) \tilde{W}(y,s)] = (x\wedge y)(t \wedge s).
\]
Apart from the global well-posedness
in $L^2(\mathbb{R})$, it is known that an
invariant measure for the transition semigroup exists,
for example, via an argument of \cite{CF1994,Fla1994} by using the ergodic theorem \cite{DDT1994, GN1999}.
Similar techniques have also been applied to
study the two-dimensional Navier-Stokes equations driven by space-time white noises ({\it e.g.} \cite{DD2002}).

Attention in the development of the stochastic Burgers equation with vanishing viscosity
has also been turned to the question of additive (spatial) noise in an equation of the form:
\begin{equation}\label{3.1a}
\pd_t u + \pd_x \big(\frac{u^2}{2}\big)
= \sum_{k = 0}^\infty \pd_x F_k(x) \d W^k + \ep \pd^2_{xx} u.
\end{equation}
The existence of invariant measures for equation \eqref{3.1a} with $\ep=0$ is known ({\it e.g.} \cite{EKMS2000}).
One of the key points is that there is enough energy dissipation in the inviscid limiting solutions
as $\ep \to 0$ (satisfying the Lax entropy condition) so that such an invariant measure exists.

The argument for the existence proof of invariant measures in \cite{EKMS2000}
is not directly via the general methods discussed in \S 2 above.
Instead, the structure of the equation is exploited to form a variational problem in \cite{EKMS1997}.
The minimizers of the action functional
\begin{align*}
\mathcal{A}[y(t)] = \frac{1}{2}\int_{t_1}^{t_2} \dot{y}^2(s) \;\d s + \int_{t_1}^{t_2}\sum_k F_k(y(s))\,\d W^k(s)
\end{align*}
are the curves that satisfy Newton's equations for the characteristics.
These minimizers have an existence and uniqueness property with probability one.
Through this, a {\it one force-one solution} principle has been shown, in which
the random attractor consists of a single trajectory almost surely,
which in turn leads to the proof of the existence of an invariant measure for \eqref{3.1a}.

\subsection{Kinetic formulation}

The theory of kinetic formulation has been developed over the last three decades
({\it cf.} Perthame \cite{Per2002} and the references cited therein).
In particular, the compactness of entropy solutions of multidimensional scalar hyperbolic conservation laws
with a genuine nonlinearity was first established by Lions-Perthame-Tadmor in \cite{LPT1994a} via
combining the kinetic formulation with corresponding velocity averaging.
The velocity averaging is a technique whereby a genuine nonlinearity condition ({\it i.e.}
a non-degeneracy condition on the nonlinearity)
can be shown to imply the compactness (or even improved fractional regularity under a stronger condition)
of solutions via the kinetic formulation,
as seen in subsequent sections, especially in  condition \eqref{3.14a}.

We discuss the kinetic formulation in the context of scalar hyperbolic conservation laws here.

One of the inspirations for a kinetic formulation originated from the kinetic theory of gases.
One starts with a simple step function as the {\it kinetic function}:
\[
\chi^r(\xi):=\chi(\xi,r)
= \begin{cases}
  \,1 &\quad\mbox{for $0 < \xi < r$}, \\
  \,-1 &\quad\mbox{for $r < \xi< 0$},\\
  \, 0 &\quad \mbox{otherwise}.
\end{cases}
\]
Then, for any $\eta\in C^1$,
the following representation formula holds:
\begin{align}\label{eq:representation_kinetic}
\int_{\mathbb{R}} \eta'(\xi) \chi^u(\xi)\;\d \xi =\eta(u)-\eta(0).
\end{align}
A simple combination of kinetic functions yields
\begin{align}\label{eq:kinetic_combination1}
|u - v| =  \int \big(|\chi^u| + |\chi^v| - 2\chi^u\chi^v\big)\;\d \xi.
\end{align}
This provides an approach to the derivation of the $L^1$--contraction between two solutions, by estimating the terms on the right.

\smallskip
There are several variations on the form of the kinetic function.
Since $|u - v|=(u - v)_+ + (v-u)_+$, it suffices for a variation, or combinations, of
the kinetic function to capture $(u - v)_+$, which is simpler than (\ref{eq:kinetic_combination1}).
This can be done by considering the following kinetic function:
\[
\tilde{\chi}^u:=\tilde{\chi}(\xi,u) = 1 - H(\xi - u) = H(u - \xi),
\]
where $H = \mathbf{1}_{[0,\infty)}$ is the Heaviside step function.
We then have the representation formula:
\[
\eta(u) = \int_{\mathbb{R}} \eta'(\xi) \tilde{\chi}^u(\xi) \;\d \xi  \,\, \qquad \mbox{for $\eta \in C^1$ with $\eta(-\infty) = 0$}.
\]
In particular,
\begin{equation}\label{3.4a}
(u - v)_+ = \int \tilde{\chi}^u(\xi) \big(1 - \tilde{\chi}^v(\xi)\big)\;\d\xi.
\end{equation}
Such a kinetic function has been popularized by \cite{LPT1994a}
and has been used, inter alia, in \cite{ChenPang-1,DV2010,DHV2014,DV2015}, and even as far back as \cite{GM1989}.

The usefulness of the kinetic function
can be seen in the {\it kinetic formulation} of scalar conservation laws,
in which the kinetic variable takes the place of the solution
in the nonlinear coefficients so that a degree of linearity
is restored for analysis.
In this formulation, many powerful linear methods such as the Fourier transform
become not only applicable, but also natural.

\smallskip
Following Chen-Pang \cite{ChenPang-1},
we now derive the {\it kinetic formulation} of  nonlinear anisotropic
parabolic-hyperbolic equations of second order:
\begin{equation}\label{3.5a}
\pd_t u +\nabla \cdot F(u) = \nabla \cdot (\A(u)\nabla u) + \sigma(u, x) \pd_t W,
\end{equation}
where $F$ is a locally Lipschitz vector flux function of polynomial growth,
$\A$ is a positive semi-definite matrix function with continuous entries of polynomial growth,
and $\nabla=\nabla_x:=(\partial_{x_1}, \cdots, \partial_{x_d})$.

\smallskip
Consider the vanishing viscosity approximation to \eqref{3.5a}:
\[
\pd_t u^\ep + \nabla \cdot F(u^\ep)
= \nabla \cdot \big((\A(u^\ep) + \ep \mathrm{I})\nabla u^\ep\big) + \sigma(u^\ep,x) \pd_t W,
\]
where $I$ is the identity matrix.
Let $\eta\in C^1$ be an entropy with $\eta(0) = 0$.
Using the Ito formula,
we have
\begin{align*}
\pd_t \eta(u^\ep)
= & -\eta'(u^\ep) \nabla \cdot F(u^\ep) + \eta'(u^\ep) \sigma(u^\ep, x) \pd_t W+ \frac{1}{2} \eta''(u^\ep) \sigma^2(u^\ep, x)\\
&+ \nabla \cdot\big(\eta'(u^\ep) \A(u^\ep) \nabla u^\ep\big) -  \eta''(u^\ep) \A(u^\ep): \big(\nabla u^\ep \otimes \nabla u^\ep\big)\\
& + \ep \Delta \eta(u^\ep)  - \ep \eta''(u^\ep) |\nabla u^\ep|^2,
\end{align*}
where we have used the notation:
$\A : \B=\sum_{i,j} \mathbf{a}_{ij} \mathbf{b}_{ij}$
for  matrices $\A = (\mathbf{a}_{ij})$ and $\B = (\mathbf{b}_{ij})$ of the same size.

Applying the representation formula (\ref{eq:representation_kinetic}) yields
\begin{align*}
\pd_t \int \eta'(\xi) \chi^{u^\ep} \;\d \xi
= & -\nabla \cdot \Big(\int \eta'(\xi) F'(\xi) \chi^{u^\ep} \;\d \xi\Big)
 +  \langle \sigma(\cdot,x)\pd_t W(t)\delta(\cdot - u^\ep),\, \eta'(\cdot)\rangle \\
& +\nabla^2 : \Big(\int \eta'(\xi)\A(\xi) \chi^{u^\ep} \;\d \xi\Big)
   -\langle \A(\cdot) : (\nabla u^\ep \otimes \nabla u^\ep) \delta(\cdot - u^\ep),\, \eta''(\cdot)\rangle\\
&-\langle \ep |\nabla u^\ep|^2\delta(\cdot - u^\ep), \; \eta''(\cdot)\rangle
   + \frac{1}{2} \langle \sigma^2(\cdot, x) \delta(\cdot - u^\ep), \;  \eta''(\cdot)\rangle\\
&+ \ep \Delta \Big(\int \eta'(\xi) \chi^{u^\ep} \;\d \xi\Big).
\end{align*}

Assume that $u^\ep(x,t)\to u(x,t)$ {\it a.e.} almost surely as $\ep\to 0$.
Then, taking $\eta'(\xi)$ as  a test function and letting $\ep \to 0$, we arrive heuristically at the formulation:

\medskip
\begin{align}\label{eq:kinetic_formulation}
\pd_t \chi^u + F'(\xi) \cdot \nabla \chi^u = \A(\xi) : \nabla^2 \chi^u +  \sigma(\xi, x) \pd_t W(t)\delta(\xi - u)  + \pd_\xi( m^u + n^u - p^u),
\end{align}

\medskip
\noindent
which holds in the distributional sense, where $m^u$, $n^u$, and $p^u$ are Radon measures that are
the limits of the following measure sequences:
\begin{align*}
&\ep |\nabla u^\ep|^2 \delta(\xi - u^\ep) \rightharpoonup m^u,\\
&\A(\xi) : \big(\nabla u^\ep \otimes \nabla u^\ep\big) \delta(\xi -u^\ep) \rightharpoonup  n^u,\\
&\frac{1}{2} \sigma^2(\xi,x) \delta(\xi - u^\ep) \rightharpoonup  p^u.
\end{align*}
The Radon measure $m^u$ is the {\it kinetic dissipation measure}
and $n^u$ is the {\it parabolic defect measure},
which capture the dissipation from the vanishing viscosity terms and the degenerate parabolic terms, respectively.
In addition, the Radon measure
$$
p^u=\frac{1}{2}\sigma^2(\xi,x)\delta(\xi-u)
$$
arises from the It\^o correction.
As $\A$ is positive semi-definite, it is manifest that $m^u$, $n^u$, and $p^u$ are all non-negative.

More precisely, the parabolic defect measure $n^u\ge 0$ is determined by the following: For any $\varphi\in C_0(\R\times \R^d\times \R_+)$,
\begin{equation}\label{3.6a}
n^u(\varphi)=\int_{\R_+}\int_{\R^d}\varphi(u(x,t),x,t)\big|\nabla_x\cdot\big(\int_0^u\ab(\zeta)d\zeta\big)\big|^2\,\d x\, \d t .
\end{equation}
The kinetic dissipation measure $m^u\ge 0$ satisfies the following:
\begin{enumerate}
\item[\rm (i)] For $B_R^c\subset \R$ as the complement of the ball of radius $R$,
\begin{align}\label{3.6b}
\lim_{R\to\infty}\Ex\big[(m^u+n^u)(B^c_R\times \T^d \times [0,T])\big] = 0;
\end{align}
\item[\rm (ii)] For any $\varphi \in C_0(\mathbb{R}\times\R^d)$,
\begin{align}\label{3.6c}
\int_{\mathbb{R}\times \R^d \times [0,T]} \varphi(\xi, x)\, \d (m^u+n^u)(\omega; \xi, x, t ) \in L^2(\Omega)
\end{align}
admits a predictable representative (in the $L^2$--equivalence classes of functions).
\end{enumerate}

Then, following Chen-Pang \cite{ChenPang-1}, we introduce the notion of kinetic solutions:

\begin{defin}[Stochastic kinetic solutions]\label{def:ksolution}
A function
\[
u \in L^p(\Omega \times [0,T]; L^p(\R^d)) \cap L^p(\Omega; L^\infty([0,T]; L^p(\R^d)))
\]
is called a \emph{kinetic solution} of \eqref{3.5a} with initial data: $u|_{t=0}=u_0$, provided that
$u$ satisfies the following:

\smallskip
\begin{itemize}
\item[(i)]
$\nabla \cdot \big(\int_0^u \ab(\xi) \; d \xi\big) \in L^2(\Omega \times \R^d\times [0,T])$;

\smallskip
\item[(ii)] For any bounded $\varphi \in C(\mathbb{R})$, the Chen-Perthame chain rule relation in \cite{CP2003} holds:
\begin{align}\label{eq:art_chainrule}
\nabla \cdot \Big(\int_0^u \varphi(\xi) \ab(\xi) \; \d \xi\Big) = \varphi(u)\;\nabla \cdot \Big(\int_0^u \ab(\xi) \;\d \xi\Big)
\end{align}
in $\mathcal{D}'(\mathbb{T}^d)$ and almost everywhere in $(t,\omega)$.

\smallskip
\item[(iii)]
There is a kinetic measure $m^u\ge 0$ $\mathbb{P}$-{\it a.e.} such that, given the parabolic defect measure $n^u$,
the following holds almost surely{\rm :}  For any $\varphi \in C^\infty_c(\mathbb{R}, \R^d \times [0,T))$,
\begin{align}
&-\int_0^T \iint \chi(\xi, u) \; \pd_t \varphi\, \d \xi\,\d x\,\d t
  - \iint \chi(\xi, u_0)\,\varphi(\xi,x,0) \;\d \xi\,\d x\notag\\
&= \int_0^T \iint \chi(\xi, u)\,F'(\xi)\cdot \nabla \varphi  \;\d\xi\,\d x\,\d t
   + \int_0^T \iint \chi(\xi, u)   \A(\xi):\nabla^2\varphi \;\d \xi\,\d x\,\d t \notag\\
&\quad  + \int_0^T \iint \varphi_\xi \, \d (m^u+n^u)(\xi,x,t)
  -  \frac{1}{2} \int_0^T \int \varphi_u(u,x,t)\sigma^2(u,x)\;\d x\,\d t\notag\\
&\quad  - \int_0^T \int \varphi(u,x,t)\sigma(u,x) \;\d x\,\d W \qquad\,\,\mbox{almost surely}.
\label{eq:definsol}
\end{align}
\end{itemize}
\end{defin}

Equation \eqref{eq:definsol} is obtained by testing \eqref{eq:kinetic_formulation}
with $\varphi$ and using the chain rule (\ref{eq:art_chainrule}).

\smallskip
\subsection{General scalar hyperbolic conservation laws driven by stochastic forcing}

In Feng-Nualart \cite{FN2008}, the well-posedness was studied for the one-dimensional scalar conservation laws driven by white noise:
\begin{align*}
\pd_t u + \pd_x F(u) = \int_{z\in Z}\sigma(u,x;z)\,\d_z W(t, z),
\end{align*}
where $Z$ is a metric space, and $W$ is a space-time Gaussian noise martingale random measure with respect to a filtration $\{\mathscr{F}_t\}_{t \ge 0}$ satisfying
\[
\Ex\big[W(t,A) \cap W(t,B)\big] = \mu(A \cap B) t
\]
for measurable sets $A,B\subset Z$, with a $\sigma$-finite Borel measure $\mu$ on $Z$.
The well-posedness theory was developed around the notion
of {\it strong stochastic entropy solutions} introduced in Definition 2.6 in \cite{FN2008}
when $t\in [0,T)$ for any fixed $T\in (0, \infty)$.
In addition to the usual definition of entropy solutions,
the following further conditions on the solution, $u=u(x,t)$, for $t\in [0,T]$ are required:

For any smooth approximation function $\beta(u)$ of function $u_+$ on $\mathbb{R}$
and any $\varphi\in C^\infty(\mathbb{R}^d\times \mathbb{R}^d)$ with $\varphi \geq 0$,
and for any $\mathscr{F}_t$--adapted function $v$ satisfying
$\sup_{0 \leq t \leq T} \Ex[\|v\|_{L^p_x}^p] < \infty$,
there exists a deterministic function $\{A(s,t) : 0 \leq s \leq t\}$ such that the functional
\begin{align*}
f(r,z,u,y):= \int_{\mathbb{R}^d} \beta'(v(x,r) - u) \sigma(v(x,r),x;z)\varphi(x,y) \;\d x
\end{align*}
satisfies
\begin{align*}
&\Ex\big[\int \int_{(s,t]\times Z} f(r,z,u(y,t),y) \d W(r,z) \,\d y\big]\\
&\leq \Ex\big[\int_{(s,t]\times Z}  \int \pd_v f(r,z,v(y,r),y) \sigma(y,u(y,r);z) \;\d y\,\d r\d \mu(z)\big] + A(s,t),
\end{align*}
and that there is a sequence of partitions of $[0,T]$ so that
\[
\lim_{\max|t_{i + 1} - t_i| \to 0} \sum_{i = 1}^m A(t_i,t_{i + 1}) = 0.
\]

This notion of a solution addresses the problem that, in any direct adaptation of the deterministic notion of entropy solutions,
one encounters the question of adaptiveness of the It\^o integral in the noise-noise interaction.
With this notion, in \cite{FN2008},
the $L^1$--contraction and comparison estimates of strong stochastic entropy solutions in any spatial dimension
were established, while the existence of solutions is limited to the one-dimensional case
 based on the compensated compactness argument in Chen-Lu \cite{ChenLu}.

In Chen-Ding-Karlsen \cite{CDK2012}, the existence theory for strong stochastic entropy solutions
was established for any spatial dimension with the key observation that
the following $BV$ bound is a corollary from the $L^1$--contraction inequality:
\begin{align*}
\Ex\big[\|u(t)\|_{BV}\big] \leq \Ex\big[\|u_0\|_{BV}\big],
\end{align*}
which provides the strong compactness required for the existence theory in any spatial dimension.
More precisely, the following theorem holds:

\begin{thm}
Consider the Cauchy problem of the equation{\rm :}
\begin{equation}\label{3.8a}
\pd_t u + \nabla \cdot F(u) = \sigma(u) \pd_t W
\end{equation}
with initial condition{\rm :}
\begin{equation}\label{3.8b}
u|_{t=0}=u_0,
\end{equation}
satisfying
\begin{align*}
\Ex\big[\|u_0\|_{L^p}^p + \|u_0\|_{BV}\big] < \infty  \qquad \mbox{for $p>1$},
\end{align*}
where $F$ is a locally Lipschitz function of polynomial growth and $\sigma$ is a globally Lipschitz function.
Then there exists a unique strong stochastic entropy solution of the Cauchy problem \eqref{3.8a}--\eqref{3.8b}
satisfying
\[
\Ex\big[\|u(t)\|_{BV}\big] \leq \Ex\big[\|u_0\|_{BV}\big].
\]
\end{thm}
This existence theory in $L^p \cap BV$ can also been extended to the second-order equations \eqref{3.5a}
as established in Chen-Pang \cite{ChenPang-1}, including the case with heterogeneous flux
functions $F=F(u,x)$ ({\it i.e.} the space-translational variant case).

A well-posedness theory can also be developed for kinetic solutions to the multidimensional scalar balance laws
with stochastic force \eqref{3.8a}, by employing the Gy\"ongy-Krylov framework
where the existence of a martingale solution with pathwise uniqueness guarantees the strong existence;
see \cite{DV2010}.
In particular, the existence of martingale solutions can be proved via the notion of kinetic solutions.
These results can be extended ({\it e.g.} \cite{Hof2013,DHV2014}) to encompass degenerate parabolic equations:
\[
\pd_t u + \nabla \cdot F(u) - \nabla \cdot (\A(u) \nabla u) = \sigma(u) \pd_t W.
\]

A well-posedness theory has also been established based on the viscosity solutions (such as
in \cite{BVW2012}).
To achieve this, the difficulties caused by the noise-noise interaction that has a non-zero correlation
for the multiplicative noise case
are avoided by directly comparing two entropy solutions to a viscosity solution.

In Karlsen-St\o rrensen \cite{KS2015}, these different viewpoints have been partially reconciled
via a Malliavin viewpoint,
in which the constant in the Kruzhkov entropy is interpreted
as a Malliavin differentiable variable.

\smallskip
Long-time asymptotic results concerning the existence and uniqueness of invariant measures
have followed the well-posedness theory.
Concerning the stochastic balance law:
\begin{align*}
\pd_t u + \nabla \cdot F(u) = \Phi(x) \;\d B,
\end{align*}
with evolution on torus $\mathbb{T}^d$,
where $B = \sum_k e_k W_k$ is a cylindrical Wiener process, $\{e_k\}$ is a complete orthonormal
basis of a Hilbert space, $\Phi$ is a Hilbert-Schmidt operator given by $\Phi(x) = \sum_k g_k(x) e_k$,
and $g_k(x)$ satisfies
\[
\int_{\mathbb{T}^d} g_k(x)  \;\d x = 0,
\]
the existence and uniqueness of invariant measures were shown in \cite{DV2015}.
In this case, the noise is {\it additive}; that is, it depends only on the spatial variable,
but is independent of the solution -- a point to which we will return.

These results can be summarized as follows:

\begin{thm}
Let $F$ satisfy the non-degeneracy condition{\rm :} For some $b  < 1$ and a constant $C >0$,
\begin{equation}\label{3.14a}
\delta(\ep) := \int_0^\infty e^{-t}\sup_{\tau\in \R, |\hat{k}|=1}
\mathcal{L}^1(\{\xi: |F'(\xi)\cdot \hat{k} + \tau| \le \ep t\}) \; \d t
\le C\ep^b
\end{equation}
for the Lebesgue measure $\mathcal{L}^1$ on $\mathbb{R}$,
in addition to the condition that $|F''(\xi)| \lesssim |\xi| + 1$.
Then there exits an invariant measure to the process.
Furthermore, if $|F''(\xi)|\lesssim 1$ is bounded,
then the invariant measure is unique.
\end{thm}

The bounds for the spaces on which the invariant measures are supported have also been derived.
This result has been obtained by employing the velocity averaging.
It has been built also on the related ideas of kinetic solutions,
which is first applied to the velocity averaging
in the deterministic context.
They avoided the question of the Fourier transforms of the Wiener process
by introducing regularizing operators.

Similar results were also derived for
\[
\pd_t u + \nabla \cdot (F(u) \circ \d W) = 0,
\]
by further employing the conservative form as considered in Lions-Perthame-Souganidis \cite{LPS2013,LPS2014};
see Gess-Souganidis \cite{GS2014}.
A generalization of this with a degenerate parabolic term $\nabla \cdot (\A(u)\nabla u)$ has
also been considered in \cite{GS2016,FG-2019}. In particular, Fehrman-Gess \cite{FG-2019} investigated the well-posedness and
continuous dependence of the stochastic degenerate parabolic equations of porous medium type,
including the cases with fast diffusion and heterogeneous fluxes.

By using the methods developed in \cite{FGRW2015,KS2012,HM2008,HM2011}
and developing the probabilistic Gronwall inequality based on delicate reasoning about a stopping time,
such MHD equations driven by additive noise of zero spatial average
in the vanishing Rossby number and vanishing magnetic Reynold's number limit were also shown to
have a unique invariant measure
(that is necessarily ergodic) in \cite{FFGR2016}.

\section{Stochastic Anisotropic Parabolic-Hyperbolic Equations \\ I: Existence of Invariant Measures} \label{sec:dp_invariantm}

In this section, we present an approach for establishing the existence of invariant measures for
nonlinear anisotropic parabolic-hyperbolic equations driven by stochastic forcing:
\begin{align}\label{eq:Main2}
\pd_t u + \nabla \cdot F(u) = \nabla \cdot(\A(u) \nabla u) + \sigma(x) \pd_t W,
\end{align}
where $\A$ is positive semi-definite, and $\sigma$ has zero average over $\mathbb{T}^d$.
The main focus of this section is on the presentation of the approach, so we do not seek the optimality of the results,
while the results presented below can be further improved by refining the arguments and technical estimates
required for the approach which is out of the scope of this section.
More precisely, we establish the following theorem:

\begin{thm}
Let $F$ and $\A$ satisfy
the nonlinearity-diffusivity condition{\rm :}
There exist $\beta\in (1,2)$,
$\kappa\in (0,1)$, and $C>0$, independent of $\lambda$, such that
\begin{align}\label{decayrate}
\sup_{\tau \in \R,\, |\hat{k}|=1}\int \frac{\lambda (\A(\xi): \hat{k}\otimes \hat{k}  + \lambda)}{(\A(\xi): \hat{k}\otimes \hat{k}  +\lambda)^2
    + \lambda^{\beta}|F'(\xi)\cdot \hat{k} + \tau|^2}  \;\d \xi
=:\eta(\lambda)\le C \lambda^\kappa \to 0
\end{align}
as $\lambda\to 0$.
In addition, let $F$ and $\A$ satisfy the condition{\rm :}
\begin{align}\label{eq:conditions_existence}
|F''(\xi)| \lesssim |\xi| + 1, \qquad
|\A'(\xi)| \lesssim |\xi| +1.
\end{align}
Then there exists an invariant measure to the process associated with the solutions
to \eqref{eq:Main2}.
\end{thm}

The approach is motivated by Debussche-Vovelle \cite{DV2015} by extending
the case from first-order scalar balance laws to the second-order
degenerate parabolic-hyperbolic equation \eqref{eq:Main2}.
The first-order case is handled in \cite{DV2015}, based on the velocity averaging
and built on Lemma 2.4 of Bouchut-Desvillettes \cite{BD1999}.
In our approach, we require a modified version of this lemma,
which is incorporated into the calculation
that allows us to exploit the cancellations in an oscillatory integral in this more general case than
the first-order case.
We now proceed to prove the theorem as follows:
\begin{itemize}
\item[(i)] First we incorporate regularizing operators into the equation in order to exploit the bounds
   that can be provided in the Duhamel representation of the solution.
\item[(ii)] We separate the Duhamel representation of the solution into four different summands,
   the $W^{s,q}$ norms of which we estimate.
\item[(iii)] Adding these estimates together by the triangle inequality and using the compact inclusion of $W^{s,q}$ into a suitable $L^q$ norm
allow us to invoke the Krylov-Bogoliubov machinery described in \S \ref{sec:krylov_bogoliubov}.
\end{itemize}

We expound on the nonlinearity condition \eqref{decayrate} in a remark below.
As the conditions in (\ref{eq:conditions_existence}) are invoked along the way, we also explain their relevance.

\medskip
Consider the kinetic formulation of equation (\ref{eq:Main2}):
\begin{align}\label{4.2a}
\pd_t \chi^u +\big(F'(\xi) \cdot \nabla  - \A(\xi) :\nabla \otimes \nabla \big)\chi^u
= \pd_\xi(m^u + n^u - p^u) + \sigma(x) \delta(\xi - u) \pd_t W.
\end{align}

In order to handle the two measures: $m^u + n^u - p^u$ and $\sigma(x) \delta(\xi - u)$,
we need to regularize the operators as in \cite{DV2015},
by adding $\gamma (-\Delta)^\alpha + \theta \, I$ to each side:
\begin{align}
&\pd_t \chi^u +\big(F'(\xi) \cdot \nabla  - \A(\xi) :\nabla \otimes \nabla + \gamma(-\Delta)^\alpha  + \theta\,I\big)\chi^u \nonumber\\
&= \big(\gamma(-\Delta)^\alpha + \theta I\big)\chi^u + \pd_\xi(m^u + n^u - p^u) + \sigma(x) \delta(\xi - u) \pd_t W \label{4.3a}
\end{align}
for $\alpha=\frac{\beta-1}{\beta}\in (0, \frac{1}{2})$ for some $\beta\in (1,2)$ required in the nonlinearity-diffusivity condition
\eqref {decayrate}.

We adapt the semigroup approach.
There are specific reasons to include these regularizing operators: In order to estimate the measure, $\sigma(x) \delta(\xi - u)\pd_t W$,
we require a spatial regularization provided by $(-\Delta)^\alpha$ and temporal decay provided by $\theta I$.

More specifically, let $\mathcal{S}(t)$ be the semigroup of operator
$\pd_t +\big(F'(\xi) \cdot \nabla  - \A(\xi) :\nabla \otimes \nabla + \gamma(-\Delta)^\alpha  + \theta\,I\big)$:
\begin{align}
\se(t) f(x)=&\, e^{-(F'(\xi) \cdot \nabla -\A(\xi): \nabla \otimes \nabla + \gamma(-\Delta)^\alpha + \theta I)t} f\nonumber\\
= &\, e^{-\theta t} \big(e^{t\A(\xi) : \nabla \otimes \nabla -t\gamma(-\Delta)^{\alpha}} f\big) (x - F'(\xi) t)
\qquad\,\,\mbox{for any $f=f(x)$}. \label{4.4a}
\end{align}
Then we can express the solution, $\chi^u$, to the kinetic formulation in the mild formulation:
\begin{align*}
\chi^u =& \, \se(t) \chi^u(\xi,x,0) + \int_0^t \se(s) (\gamma(-\Delta)^\alpha - \theta I) \chi^u(\xi, x, t - s) \;\d s\\
&+\int_0^t \se(t - s) \pd_\xi(m^u + n^u - p^u) (\xi,x,s)\;\d s + \int_0^t \se(t - s) \sigma(x) \delta(\xi - u(x,s)) \;\d W_s.
\end{align*}
This leads to the decomposition:
\begin{equation}\label{4.6a}
u = u^0 + u^\flat + M_1 + M_2,
\end{equation}
where
\begin{align}
u^0(x,t) = & \int \se(t)\chi^u(\xi,x,0) \;\d \xi , \label{4.7a}\\
u^\flat(x,t) = &\int \int_0^t \se(s) (\gamma(-\Delta)^\alpha - \theta I)\chi^u(\xi,x,t - s)\;\d s\,\d \xi ,\label{4.7b}\\
\LL M_1, \varphi \RR = &  \int_0^t\int \LL \pd_\xi(m^u + n^u - p^u)(\cdot,x,t - s),\, \se^*(s)\varphi\RR\, \d x\,\d s,\label{4.7c}\\
\LL M_2, \varphi \RR = & \int_0^t\int \LL \delta(\cdot - u(x,s)),\, \se^*(t - s)\varphi\RR \sigma(x)  \;\d x \;\d W_s,\label{4.7d}
\end{align}
where $\se^*(t)$ is the dual operator of the semigroup operator $\se(t)$, and $\varphi\in C(\mathbb{T}^d)$.

We now estimate each of the four terms separately in each subsection:
The first two integrals are essentially ``deterministic'' parts
and estimated by the velocity averaging method,
and the final two integrals incorporate stochastic elements and
are treated by a kernel estimate on semigroup $\se(t)$.

\subsection{Analysis of $u^0$}
Notice that the local Fourier transform in $x\in \mathbb{T}^d$ for any periodic function $g(x,\cdot)$ in $x$  with
period $P=(P_1,\cdots, P_d)$ is:
$$
\hat{g}(k, \cdot )=
\frac{1}{|\mathbb{T}^d|}
\int_{\mathbb{T}^d} g(x,\cdot)e^{-i x\cdot k}\,\d x,
$$
where frequencies $k=(k_1,\cdots, k_d)$ are
discrete:
$$
k_i=\frac{2\pi}{P_i}n_i, \qquad n_i=0, \pm 1, \pm 2, \cdots, \,\,\, i=1,\cdots, d.
$$

Taking the Fourier transform in $x$ and integrating in $\xi$, we have
\begin{align*}
\widehat{u^0}(k,t) = &\int \hat{\se}(t) \widehat{\chi^u}(\xi,k,0) \;\d \xi \\
= & \int e^{-(iF'(\xi) \cdot k  + \A(\xi) : (k \otimes k)+ \omega_k |k|)t}\, \widehat{\chi^u} (\xi,k,0)\;\d \xi,
\end{align*}
where $\omega_k=\gamma|k|^{2\alpha-1}+\theta |k|^{-1}$.

For simplicity, we denote $\hat{k}=\frac{k}{|k|}$ and
$\mathcal{A}= \mathcal{A}(\xi,\hat{k})=\A(\xi) : \hat{k} \otimes\hat{k}$.
Then we square the above and integrate in $t$ from $0$ to $T$ to obtain
\begin{align}
\int_0^T |\widehat{u^0}(k,t)|^2 \;\d t
= & \int_0^T \left| \int e^{-(iF'(\xi) \cdot \hat{k}+ \mathcal{A}(\xi,\hat{k})|k|+\omega_k) |k| t}\, \widehat{\chi^u} (\xi,k,0)\;\d \xi \right|^2 \;\d t\notag\\
\leq &\, \frac{1}{|k|} \int \left| \int \mathds{1}_{\{s > 0\}}
 e^{-(iF'(\xi) \cdot\hat{k}+\mathcal{A}(\xi,\hat{k})|k|+\omega_k)s}\,\widehat{\chi^u} (\xi,k,0)\;\d \xi \right|^2 \;\d s.\label{eq:model_estimate1}
\end{align}

Notice that it is impossible to extract the entire non-oscillatory part of the exponential from
the integral in $\xi$, as was done with the lemma of Bouchut-Desvillettes \cite{BD1999}.
However, by extending the range of integration over all $\mathbb{R}$ to make the function in $s$ smoother
so that its transform has better decay properties, we can partially exploit the cancellations later:
\begin{align}
\int_0^T |\widehat{u^0}(k,t)|^2 \;\d t
\leq &\frac{1}{|k|}\int_{-\infty}^\infty
  \left|\int e^{iF'(\xi) \cdot \hat{k} s} e^{-(\omega_k +\mathcal{A}|k|)|s|} \widehat{\chi^u} (\xi,k,0)\;\d \xi  \right|^2 \;\d s.
  \end{align}

We can evaluate the temporal Fourier transform of the integrand explicitly:
\begin{align*}
\mathscr{F}^{-1}\left\{ e^{iF'(\xi) \cdot \hat{k} s} e^{-(\omega_k + \mathcal{A}|k|)|s|}  \right\}(\tau)
= - \frac{2(\mathcal{A}|k|+\omega_k)}{(\mathcal{A}|k|+\omega_k)^2
     + |F'(\xi)\cdot \hat{k} + \tau|^2}.
\end{align*}

Next, using the Parseval identity in the temporal variable and the Cauchy-Schwarz inequality,
we have
\begin{align*}
\int_0^T \big|\widehat{u^0}(k,t)\big|^2 \;\d t
\leq  &\, \frac{1}{|k|}\int_{-\infty}^\infty
  \left|\int e^{iF'(\xi) \cdot \hat{k} s} e^{-(\omega_k + \mathcal{A}|k|)|s|} \widehat{\chi^u} (\xi,k,0)\;\d \xi \right|^2 \;\d s\\
=  &\, \frac{1}{|k|}\int_{-\infty}^\infty
  \left|\mathcal{F}^{-1}
   \left\{\int e^{iF'(\xi) \cdot \hat{k} s} e^{-(\omega_k + \mathcal{A}|k|)|s|} \widehat{\chi^u} (\xi,k,0)\;\d \xi\right\}(\tau)\right|^2 \;\d \tau\\
=  &\, \frac{4}{|k|}\int_{-\infty}^\infty
    \left|\int \frac{\mathcal{A}|k|+\omega_k}{(\mathcal{A}|k|+\omega_k)^2
       + |F'(\xi)\cdot \hat{k} + \tau|^2} \widehat{\chi^u} (\xi,k,0)\;\d \xi  \right|^2 \;\d \tau\\
\le &\, \frac{4}{|k|}\int_{-\infty}^\infty
  \left( \int \big|\widehat{\chi^u}(\xi,k,0)\big|^2\frac{\mathcal{A}|k|+\omega_k}{(\mathcal{A}|k|+\omega_k)^2
    + |F'(\xi)\cdot \hat{k} + \tau|^2}\;\d \xi  \right)\,\\
    &\qquad\quad\quad\times\left(\int \frac{\mathcal{A}|k|+\omega_k}{(\mathcal{A}|k|+\omega_k)^2
    + |F'(\xi)\cdot \hat{k} + \tau|^2}  \;\d \xi\right)\d \tau\\
\le &\, \frac{4}{|k|\omega_k}
\int \big|\widehat{\chi^u}(\xi,k,0)\big|^2\left(\int \frac{\mathcal{A}|k|+\omega_k}{(\mathcal{A}|k|+\omega_k)^2
    + |F'(\xi)\cdot \hat{k} + \tau|^2} \;\d \tau\right) \d \xi\\
    &\qquad \times \sup_\tau\int \frac{\omega_k(\mathcal{A}|k| + \omega_k)}{(\mathcal{A}|k| + \omega_k)^2
    + |F'(\xi)\cdot \hat{k} + \tau|^2}  \;\d \xi.
\end{align*}

Notice that the integral
$$
\int \frac{\mathcal{A}|k|+\omega_k}{(\mathcal{A}|k|+\omega_k)^2
    + |F'(\xi)\cdot \hat{k} + \tau|^2} \;\d \tau
$$
is a constant for fixed $\xi$ by the translation invariance of $\d \tau$.

Now invoking \eqref{decayrate} and setting $\lambda = \frac{\omega_k}{|k|}$,
we have
$$
\int_0^T  |\widehat{u^0}(k,t)|^2\;\d t
\le \frac{C}{|k|\omega_k}\eta(\frac{\omega_k}{|k|})
\int |\widehat{\chi^u}(\xi,k,0)|^2\;\d \xi
$$
for some constant $C$ depending on $\gamma$ and $\theta$. That is,
$$
\int_0^T |k|^{1+\kappa}\omega_k^{1-\kappa} |\widehat{u^0}(k,t)|^2\;\d t
\le
C \int |\widehat{\chi^u}(\xi,k,0)|^2\;\d \xi.
$$

Since $u_0$ has null average over $\mathbb{T}^d$,
\begin{equation}\label{4.15a}
\widehat{u^0}(0,t)=\int_{\mathbb{T}^d}u^0(x,t)\, \d x
=\int\widehat{\chi^u}(\xi, 0,0)\, \d\xi
=\iint \chi^u(\xi, x,0)\,\d \xi\,\d x
=\int_{\mathbb{T}^d} u_0(x)\, \d x=0.
\end{equation}

Then, summing over all the discrete frequencies $k$ with $|k|\ne 0$,
using the Plancherel theorem again --- in space this time --- and
noting that $\omega_k\ge \gamma|k|^{2\alpha-1}$,
we have the estimate
\begin{align}\label{eq:u_0_estimate}
\int_0^T\|u\|_{H^{(1 - \alpha) \kappa + \alpha}_x}^2 \;\d t \le C\|u_0\|_{L^{1}_x}.
\end{align}

\medskip
\subsection{Analysis of $u^\flat$}\label{sec:uflat}

The calculation is similar:
\begin{align*}
&\int_0^T |\widehat{u^\flat}(k,t)|^2\;\d t  \\
&=\int_0^T \left|\int \int_0^t \hat{\se}(s) (\gamma|k|^{2\alpha} + \theta I)\,\widehat{\chi^u}(\xi,k,t - s)\;\d s\,\d \xi \right|^2 \d t \\
&= \int_0^T \left| \int_0^T \1_{\{t - s \geq 0\}} \int e^{-(iF'(\xi) \cdot k + \omega_k|k| + \A(\xi) : k\otimes k) s}
    \omega_k |k|\,\widehat{\chi^u}(\xi,k,t - s)\;\d \xi \, \d s\right|^2 \;\d t \\
&\leq  \Big(\int_0^\infty \omega_k|k| e^{-\omega_k |k| s}\;\d s\Big)\\
&\quad\,\,\,\times \int_0^T\Big(\int_0^T \Big|\int e^{-(iF'(\xi) \cdot k + \omega_k|k|/2 + \A(\xi) : k\otimes k) s}
 \sqrt{\omega_k |k|}\, \widehat{\chi^u}(\xi,k, t)\,\d \xi\Big|^2\d s \Big)\d t,
\end{align*}
where we have used the Cauchy-Schwarz inequality and extended the domain of the inner temporal integration to $[0,\infty)$.

This leaves us in the exact position of Eq. \eqref{eq:model_estimate1}
with an additional temporal integral in $t$  (applied only to the kinetic function $\widehat{\chi^u}$) and an additional factor of $|k|\omega_k$.
Therefore, we can conclude as in Eq. \eqref{eq:u_0_estimate}
by using the zero-spatial average property \eqref{4.15a} and  $\omega_k\le (\gamma+\theta)|k|^{2\alpha-1}$
 that
\begin{align}\label{eq:ubestimate}
\int_0^T\|u^\flat(t)\|_{H^{(1 - \alpha)\kappa}_x}^2 \;\d t\leq &C \int_0^T \|u(t)\|_{L^{1}_x}\;\d t.
\end{align}

\begin{rem}\label{rem:nondegeneracy}
Condition \eqref{decayrate} is reminiscent of the nonlinearity condition given in the deterministic setting by Chen-Perthame \cite{CP2009}.
If we discard the regularising operator $(-\Delta)^{2\alpha}$ in \eqref{4.3a}, {\it i.e.}
by setting $\alpha = 0$, then $\beta=1$
in \eqref{decayrate}.
On the other hand, we can choose $\beta$ sufficiently close to $2$ so that \eqref{decayrate} holds,
by selecting $\alpha$ close to $\frac{1}{2}$.
For both cases,
we are  able to conclude that the $u^\flat$-part of the solution operator is compact.
However, as we will see below, the regularizing effect of $(-\Delta)^{2\alpha}$ is crucial in
estimating \eqref{4.7c}--\eqref{4.7d} in the way as we do, via \eqref{eq:kernel_estimate},
in the next subsections.
As the two terms \eqref{4.7c} and \eqref{4.7d} arise from the martingale and the It\^{o} approximation, respectively,
this decay requirement beyond $o(1)$ does not appear in the deterministic setting.
\end{rem}

\smallskip
\subsection{Analysis of $M_1$}

Next we turn to the analysis of the two measures $M_1$ and $M_2$.
For this, we follow \cite{DV2015} closely,
since the only difference is the parabolic defect measure,
which has the same sign as the kinetic dissipation measure,
and the magnitude of the kinetic dissipation measure
is never invoked in \cite{DV2015}.
In this and the following sections,
we repeatedly apply bound (\ref{eq:kernel_estimate}) in order
to pursue the compactness estimates.

From \eqref{4.4a},
we see
\begin{align}\label{eq:pd_xi_se}
&\pd_\xi\big(\se^*(t -s) h(\xi,x)\big) \notag\\
&= (t - s) F''(\xi) \cdot \nabla (\se^*(t - s)h) + \A'(\xi) : \nabla^2 (\se^*(t - s) h)
   +\se^*(t - s) \pd_\xi h.
\end{align}
Then we have
\begin{align}
\LL M_1, \varphi \RR
&=   -\int_0^t\iint  \pd_\xi(\se^*(s)\varphi)\; \d (m^u + n^u - p^u)(\xi,x,t - s)
\notag\\
&= \int_0^t\iint (t - s)F''(\xi) \cdot \nabla (\se^*(t - s)\varphi)\;\d (m^u + n^u - p^u)(\xi,x,t - s)
\notag\\
&\quad + \int_0^t\iint \A'(\xi) : \nabla^2 (\se^*(t - s) \varphi)\; \d (m^u + n^u - p^u)(\xi,x,t - s).
 \label{eq:M1_expanded}
\end{align}

Now we show the following total variation estimate.

\begin{lma}\label{thm:TVestimate1}
Let $u: \mathbb{T}^d \times [0,T] \times \Omega$ be a solution with initial data $u_0$.
Let $\psi \in C_c(\mathbb{R})$ be any nonnegative and compactly supported continuous function,
and $\Psi = \int_0^s \int_0^r \psi(t)\;\d t \d r$.
Then
\begin{align*}
\Ex\big[\int_{\mathbb{T}^d \times [0,T] \times \mathbb{R}} \psi(\xi)\;\d|m^u + n^u - p^u|(\xi,x,t)\big]
 \leq  D_0 \Ex\big[\|\psi(u)\|_{L^1_{x,t}}\big] + \Ex\big[\|\Psi(u_0)\|_{L^1_x}\big],
\end{align*}
where $D_0:=\|\sigma^2\|_{L^\infty(\mathbb{T})}$.
\end{lma}

\begin{proof}
The proof is the same as that in \cite{DV2015} and involves bounding $|m^u + n^u - p^u|\leq m^u + n^u + p^u$, so that
\begin{align*}
&\Ex\big[\int_0^T \iint \psi(\xi)\; \d\,|m^u + n^u - p^u|(\xi, x, t)\big]\\
&\leq \Ex\big[\int_0^T \iint \psi(\xi)\; \d(m^u + n^u - p^u)(\xi, x,t)\big]
  +2 \Ex\big[\int_0^T \iint \psi(\xi) \;\d p^u(\xi, x, t)\big]\\
&= \Ex\big[-\int\Psi(u)\;\d x \bigg|^T_0\big]
+ \Ex\big[\int_0^T\int \sigma^2(x) \psi(u) \;\d x \d t \big],
\end{align*}
by using the kinetic equation in the sense of \eqref{eq:definsol}.
Now, using the non-negativity of $\psi$, we have
\[
\Ex\big[\int_0^T \iint \psi(\xi) \; \d |m^u + n^u - p^u|(\xi, x,t)\big]
\leq  \Ex\big[ \int\Psi(u_0)\;\d x\big] + D_0 \Ex\big[\int_0^T\int \psi(u) \;\d x\,\d t \big].
\]
\end{proof}

This estimate is quite crude, as one does not take the cancellation between measures $m^u + n^u$ and $p^u$,
both non-negative, into account.
Since there is no available way to quantify $m^u + n^u$, this is the best possible at the moment.

In addition to a total variation estimate, we also require the kernel estimate:
\begin{align}
\Big\|(-\Delta)^{\frac{\hat{\beta}}{2}}e^{(\A:\nabla \otimes \nabla - \gamma(-\Delta)^{\alpha})t}\Big\|_{L^p_{x,\xi} \to L^q_{x,\xi}}
\leq C(\gamma t)^{-\frac{d}{2\alpha} \left(\frac{1}{p}- \frac{1}{q}\right) - \frac{\hat{\beta}}{2\alpha}}.\label{eq:kernel_estimate}
\end{align}
The reason for the no additional improvement over the estimate for operator $e^{t\A:\nabla \otimes \nabla}$
is that we have not specified how degenerate $\A$ is --- it may well be simply the zero matrix.
It is the use of this kernel estimate that necessitates the inclusion of the regularizations $\gamma(-\Delta)^{\alpha} + \theta I$.

By the kernel estimate \eqref{eq:kernel_estimate}, we have
\begin{align*}
&\|(-\Delta)^{\frac{\hat{\beta}}{2}}\nabla(\se^*(t)\vp)\|_{L^\infty_{x,\xi}}
\leq  C (\gamma t)^{-\mu}e^{-\theta\, t}\|\varphi\|_{L^{p}_x},\\
&\|(-\Delta)^{\frac{\hat{\beta}}{2}}\nabla^2  (\se^*(t) \vp)\|_{L^\infty_{x,\xi}}
\leq  C(\gamma t)^{-\mu-\frac{1}{2\alpha}}
e^{-\theta\,t}\|\varphi\|_{L^{p}_x},
\end{align*}
where
$\mu:=\frac{\hat{\beta} + 1}{2\alpha} + d(\frac{1}{2\alpha} - \frac{1}{2\alpha p'})$ for $p'>1$,
and the universal constant $C$ is independent of $\gamma$ and $\theta$.

Inserting these estimates into (\ref{eq:M1_expanded}), we have the estimate:
\begin{align*}
&\Ex\big[\int_0^T \LL (-\Delta)^{\frac{\hat{\beta}}{2}}M_1, \varphi \RR\;\d t \big]\\
&=  \Ex\big[\int_0^T\int_0^t\iint (-\Delta)^{\frac{\hat{\beta}}{2}}F''(\xi) \cdot \nabla (\se^*(t - s)\varphi)\;\d(m^u + n^u - p^u)(\xi,x,t - s) \, \d t \\
& \qquad\,\, + \int_0^T\int_0^t\iint (-\Delta)^{\frac{\hat{\beta}}{2}}\A'(\xi) : \nabla^2 (\se^*(t - s) \varphi)\;\d (m^u + n^u - p^u)(\xi,x,t - s) \;\d t \big]\\
&\leq  \Ex\big[\int_0^T\int\|(-\Delta)^{\frac{\hat{\beta}}{2}}\nabla (\se^*(t - s)\varphi)\|_\infty |F''(\xi)|(t - s)\;\d |m^u + n^u - p^u|(\xi, x, s)\;\d t \big]\\
&\quad + \Ex\big[\int_0^T\int\|(-\Delta)^{\frac{\hat{\beta}}{2}}\nabla^2 (\se^*(t - s)\varphi)\|_\infty |\A'(\xi)| \;\d |m^u + n^u - p^u|(\xi, x,s)\;\d t \big].
\end{align*}

By the presence of factor $e^{-\theta(t - s)}$, we can also bound the outer temporal integral by using the definition of the Gamma function:
$$
\Gamma(z) =\int_0^\infty x^{z - 1}e^{-x} \;\d x
$$
so that, taking $\LL \cdot, \cdot \RR$ as the $L^{p'}(\T^d)$--$L^{p}(\T^d)$ pairing,
\begin{align*}
&\Ex\big[\int_0^T \LL (-\Delta)^{\frac{\hat{\beta}}{2}}M_1, \varphi \RR\;\d t \big]\\
&\leq  \int_0^T (\gamma \tau )^{ - \mu} e^{-\theta \tau} \;\d \tau\;
 \Ex\big[\int_{\mathbb{R}\times \mathbb{T}^d\times[0,T]} \|\varphi\|_{L^p_x}|F''(\xi)|\;\d |m^u + n^u - p^u|(\xi, x,s)\big]\\
&\quad +\int_0^T (\gamma \tau )^{ - \mu - \frac{1}{2\alpha}} e^{-\theta \tau} \;\d \tau\;\Ex\big[\int_{\mathbb{R}\times \mathbb{T}^d\times[0,T]}
  \|\varphi\|_{L^p_x}|\A'(\xi)|\;\d  |m^u + n^u - p^u|(\xi, x,s)\big]\\
&\leq C \theta^{\mu + 1} \gamma^{-\mu} |\Gamma(-\mu + 1)|
  \;\Ex\big[\int_{\mathbb{R}\times \mathbb{T}^d\times[0,T]} \|\varphi\|_{L^p_x}|F''(\xi)|\;\d |m^u + n^u - p^u|(\xi, x, s)\big]\\
&\quad + C \theta^{\mu - 1 - \frac{1}{2\alpha}} \gamma^{-\mu - \frac{1}{2\alpha}}\big|\Gamma(-\mu + 1 - \frac{1}{2\alpha})\big|\\
&\qquad\,\times  \;\Ex\big[\int_{\mathbb{R}\times \mathbb{T}^d\times[0,T]} \|\varphi\|_{L^p_x}|\A'(\xi)|\;\d |m^u + n^u - p^u|(\xi, x,s)\big].
\end{align*}

By duality, the total variation estimate, and the sublinearity of $F''$ and $\A'$, we have
\begin{align}
&\Ex\big[\|M_1\|_{L^1_t W^{\hat{\beta},p'}_x}\big] \notag\\
&\leq  C\theta^{\mu + 1} \gamma^{-\mu} |\Gamma(-\mu + 1)|
  \;\Ex\big[\int_{\mathbb{R}\times \mathbb{T}^d\times[0,T]}|F''(\xi)|\; \d |m^u + n^u - p^u|(\xi, x, s)\big]\notag\\
&\quad +  C\theta^{\mu - 1 - \frac{1}{2\alpha}} \gamma^{-\mu - \frac{1}{2\alpha}}\big|\Gamma(-\mu + 1 - \frac{1}{2\alpha})\big|\;
\Ex\big[\int_{\mathbb{R}\times \mathbb{T}^d\times[0,T]}|\A'(\xi)|\;\d |m^u + n^u - p^u|(\xi, x, s)\big]\notag\\
&\leq  C\Big(\theta^{\mu + 1} \gamma^{-\mu} |\Gamma(-\mu + 1)|
  +    \theta^{\mu - 1 - \frac{1}{2\alpha}} \gamma^{-\mu - \frac{1}{2\alpha}}\big|\Gamma(-\mu + 1 - \frac{1}{2\alpha})\big| \Big) \notag\\
&\qquad \times \Big(1 + \int_0^T\Ex\big[\|u(t)\|_{L^1_x}\big]\;\d t  + \Ex\big[\|u_0\|_{L^3_x}^3\big]\Big),
 \label{eq:M1estimate}
\end{align}
where we have chosen $\gamma$ and $\theta$ such that
\begin{align*}
C\Big(\theta^{\mu + 1} \gamma^{-\mu} |\Gamma(-\mu + 1)|
  +    \theta^{\mu - 1 - \frac{1}{2\alpha}} \gamma^{-\mu - \frac{1}{2\alpha}}\big|\Gamma(-\mu + 1 - \frac{1}{2\alpha})\big| \Big)
  \leq \frac{\epsilon_0}{2},
\end{align*}
for sufficiently small $\epsilon_0$ to be determined later.

\subsection{Analysis of $M_2$}

\begin{align*}
\LL M_2, \varphi \RR = \int_0^t\int \LL \delta(\cdot- u(x, s)), \varphi\,(\se(t - s)\sigma(x))\RR\,\d x \;\d W_s.
\end{align*}

We again invoke the kernel estimate.
In fact, it is here that the kernel estimate becomes indispensable.
In the stochastic setting, with a forcing term given by $\sigma(x) \delta(\xi - u(x,t)) \pd_t W$,
which does not easily lend itself to the space-time Fourier transform,
one may not simply take the Fourier transform on both sides so that
the factor, $i(\tau +  F'(\xi) \cdot k) + \A(\xi) : (k\otimes k)$, on the left side can simply be divided out,
with a certain genuine nonlinearity ({\it i.e.} the non-degeneracy condition; {\it cf.} \cite{CF1998,LPT1994a,TT2007}).
Thus, we have to find a different way to handle the forcing term.

Expanding the effect of the semigroup, we have
\begin{align*}
\LL M_2, \varphi \RR =  \int_0^t \int_{\mathbb{T}^d} e^{-\theta (t - s)} \varphi
e^{-(\A(\xi) : \nabla \otimes \nabla +\gamma(-\Delta)^{\alpha})(t - s)} \sigma(x - F'(u(x,s)) (t - s)) \;\d x\,\d W_s.
\end{align*}
Since $\sigma$ is bounded in $\mathbb{T}^d$, we see that $\sigma(\cdot - F'(u(\cdot,s))(t - s))$ is bounded in $x$.

The kernel estimate then gives
\begin{align*}
&\|e^{-\theta (t - s)} \varphi e^{-(\A(u) : \nabla \otimes \nabla + \gamma(-\Delta)^{\alpha})(t - s)}  \sigma(\cdot- F'(u(\cdot,s))(t - s))\|_{H^{\hat{\beta}}_x}\\
& \leq C \big(\gamma(t - s)\big)^{-\frac{\hat{\beta}}{2\alpha}} \|\sigma\|_{L^2_x},
\end{align*}
just as in \cite{DV2015}.
In the same way, we have
\begin{align*}
\Ex\big[\big\|\int_0^t \LL \delta(\cdot - u(x,s)), \,\se(t - s) \sigma(x) \RR \,\d W_s\big\|_{H^{\hat{\beta}}_x}^2\big]
\leq C
\gamma^{-\frac{\hat{\beta}}{\alpha}} \theta^{\frac{\hat{\beta}}{\alpha} - 1} \big|\Gamma(1 - \frac{\hat{\beta}}{\alpha})\big|.
\end{align*}
Then we have
\begin{align*}
\Ex\big[\|M_2\|_{H^{\hat{\beta}}_x}^2\big]
\leq C
\gamma^{-\frac{\hat{\beta}}{\alpha}} \theta^{\frac{\hat{\beta}}{\alpha} - 1} \big|\Gamma(1 - \frac{\hat{\beta}}{\alpha})\big|.
\end{align*}

\subsection{Completion of the existence proof}

From (\ref{eq:u_0_estimate})--(\ref{eq:ubestimate}), we have
\begin{align*}
\Ex\big[\|u^0 + u^\flat +M_2\|_{L^2_tW^{s,q}_x}^2\big]
\leq  \Ex\big[\|u(0)\|_{L^1_x}\big] + \Ex\big[\|u\|_{L^{1 }([0,T],L^{1 }_x)}\big] + CT,
\end{align*}
where $q>1$ and $s>0$.

By the standard Jensen and Young inequalities, we obtain
\begin{align*}
&\frac{1}{T}\Ex^2\big[\|u^0 + u^\flat + M_2 \|_{L^1_tW^{s,q}_x}\big]
  \leq \Ex\big[\|u^0 + u^\flat +M_2\|_{L^2_tW^{s,q}_x}^2\big],
\end{align*}
so that
$$
\Ex^2\big[\|u^0 + u^\flat + M_2 \|_{L^1_tW^{s,q}_x}\big]
\leq  C T \Big(\Ex\big[\|u(0)\|_{L^{1}_x}\big] +\Ex\big[\|u\|_{L^{1}([0,T],L^{1 }_x)}\big] + T\Big).
$$
Then we have
$$
\Ex\big[\|u^0 + u^\flat + M_2 \|_{L^1_tW^{s,q}_x}\big]
\leq  C\Big(\Ex\big[\|u(0)\|_{L^{1 }_x}\big] + T\Big)
+ \frac{\epsilon_0}{2} \Ex\big[\|u\|_{L^{1 }([0,T],L^{1 }_x)}\big].
$$

From (\ref{eq:M1estimate}), we further have
\begin{align*}
\Ex\big[\|M_1\|_{L^1([0,T],W^{\hat{\beta},p'}_x)}\big]
\leq\frac{\epsilon_0}{2}\Big(1 + \Ex\big[\|u\|_{L^{1}([0,T],L^{1}_x)}\big] + \Ex\big[\|u_0\|_{L^3_x}^3\big]\Big).
\end{align*}

By the continuous embedding $W^{s,q}_x \hookrightarrow L^{1}_x$,
\begin{align}\label{eq:wsp_bound}
\Ex\big[\|u\|_{L^1([0,T],W^{s,q}_x)}\big]
 \leq  C(\alpha,\hat{\beta},\gamma,\theta)
 \Big(1 + \Ex\big[\|u(0)\|_{L^3_x}^3\big] + T\Big).
\end{align}

Since $W^{s,q}$ is compactly embedded in $L^1$ for $q \geq 1$,
the Krylov-Bogoliubov mechanism (\S 2.2)
leads to the existence of an invariant measure.

\section{Stochastic Anisotropic Parabolic-Hyperbolic Equations \\ II:  Uniqueness of Invariant Measures}

In this section, we prove the uniqueness of invariant measures for the second-order
nonlinear stochastic equations \eqref{eq:Main2}.

\begin{thm}
Let $F$ and $\A$ satisfy
the non-degeneracy condition \eqref{decayrate}
and the boundedness condition{\rm :}
\begin{align}
|F''(\xi)| \lesssim  1,\qquad
|\A'(\xi)| \lesssim  1. \label{eq:conditions_uniqueness}
\end{align}
Then the invariant measure established in Theorem {\rm 4.1} is unique.
\end{thm}

To show the uniqueness, we first show that the solutions enter a certain ball in $L^1_x$
in finite time almost surely.
Then we show that the solutions, starting on a fixed ball, enter arbitrarily small balls almost surely,
if the noise is sufficiently small in $W^{1,\infty}$.
This allows us to conclude that
any pair of balls enters an arbitrarily small ball of one another,
since the noise is sufficiently small for any given duration with positive probability.
This is the property of recurrence discussed in the coupling method in \S 2,
which implies the uniqueness of invariant measures.
In showing the recurrence, we follow \S 4 of \cite{DV2015} quite closely.

\subsection{Uniqueness I: Finite time to enter a ball}\label{sec:uniquenessI}

The following lemma is proved in the same way as in \cite{DV2015},
via a Borel-Cantelli argument.

\begin{lma}\label{thm:enterball_kappa}
There are both a radius $\hat{\kappa}$ {\rm (}depending on the initial conditions{\rm )}
and an almost surely finite stopping time $\mathcal{T}$ such that a solution enters $B_{\hat{\kappa}}(0) \subseteq L^1(\mathbb{T}^d)$ in time $\mathcal{T}$.
\end{lma}

The proof uses the coupling method, where $v$ is another solution to the same equation with initial condition $v(0) = v_0$.
It furnishes us with the recursively defined sequence of stopping times, with $\mathcal{T}_0 = 0$ and
\begin{align} \label{eq:stopping_time}
\mathcal{T}_l = \inf\{ t \geq \mathcal{T}_{l - 1} + T \, :\,  \|u(t) \|_{L^1_x} + \|v(t)\|_{L^1_x} \leq 2 \hat{\kappa}\},
\end{align}
which are also almost surely finite.

\subsection{Uniqueness II: Bounds with small noise}\label{sec:uniquenessII}

We now prove the following key lemma for the pathwise solutions:

\begin{lma}\label{thm:small_ball_lemma}
For any $\eps > 0$, there are $T > 0 $ and $\tilde{\kappa} > 0$  such that, for the initial conditions $u_0$ satisfying
\[
\|u_0\|_{L^1_x} \leq 2\hat{\kappa},
\]
and the noise satisfying
\[
\sup_{t \in [0,T]}\|\sigma W\|_{W^{1,\infty}_x} \leq \tilde{\kappa},
\]
then
\begin{align*}
\fint_0^T \|u(t)\|_{L^1_x} \;\d t  \leq \eps,
\end{align*}
where we have used the symbol $\fint$ to denote the averaged integral.
\end{lma}

\begin{proof}
One of the differences in our estimates from \cite{DV2015} is that a kernel estimate is used
on $v^\sharp_F + v^\sharp_A$,
instead of velocity averaging techniques, since the extra derivatives are required to be handled here.
Of course, this method can also be applied to the first-order case so that the need to estimate
the average term $\fint v^\sharp \;\d x$ in \cite{DV2015} can be eliminated.
We divide the proof into nine steps.

\smallskip
1. Let $u$ be a solution of
\[
\pd_t u + \nabla \cdot F(u) + \nabla \cdot (\A(u) \nabla u) = \sigma(x) \pd_t W
\]
with initial condition $u(0) = u_0$,
and let $\tilde{u}$ be the solution to the same equation with initial condition $\tilde{u}_0$ satisfying
\begin{align*}
\|u_0 - \tilde{u}_0\|_{L^1_x} \leq \frac{\eps}{8}, \qquad \|\tilde{u}_0\|_{L^2_x} \leq C \hat{\kappa} \eps^{-\frac{d}{2}},
\end{align*}
which can be found by convolving $u_0$ with a mollifying kernel, where
$\hat{\kappa}$ is the radius constant of Lemma \ref{thm:enterball_kappa}.

\smallskip
2. Consider the difference between solution $\tilde{u}$ and noise $\sigma(x)W$:  $v = \tilde{u} - \sigma(x) W$,
which is a kinetic solution to
\begin{align*}
\pd_t v = - \nabla \cdot F(v + \sigma(x) W) + \nabla \cdot \big(\A(v + \sigma(x) W) \nabla (v + \sigma(x)W)\big).
\end{align*}

The kinetic formulation for this equation can be derived as in (\ref{eq:kinetic_formulation}):
\begin{align}
&\pd_t \chi^v + F'(\xi) \cdot \nabla \chi^v - \A(\xi) : \nabla^2 \chi^v\notag \\
&=  \big(F'(\xi) - F'(\xi + \sigma(x) W)\big) \cdot \nabla \chi^v
  - \nabla \cdot \big( (\A(\xi) - \A(\xi + \sigma(x) W)) \nabla \chi^v\big)\notag\\
&\quad  - F'(\xi + \sigma(x) W) \delta(\xi - v) \cdot \nabla (\sigma(x) W)
  + \nabla \cdot \big(\A(\xi + \sigma(x)W) \delta(\xi - v) \nabla (\sigma(x)\, W)\big)\notag\\
&\quad - \pd_\xi \big(\delta( \xi - v) \A(\xi + \sigma(x) W) :\big(\nabla (\sigma(x) W)\otimes \nabla (\sigma(x) W)\big)\big)
\label{eq:kinetic_equation2}\\
&\quad + \pd_\xi (m^v + N^v).\notag
\end{align}
A notable difference here is that the parabolic defect measure $N^v$ is not
the limit of
\[
\delta(\xi - (v^\ep + \sigma(x) W))\A(\xi) :\big(\nabla (v^\ep + \sigma W) \otimes \nabla (v^\ep + \sigma(x) W)\big),
\]
but rather the limit of
\begin{align}
N^u_\ep
= &\,\delta(\xi - v^\ep)\A(\xi + \sigma(x) W) : \big(\nabla v^\ep \otimes \nabla v^\ep\big)\notag\\
&\, + \delta(\xi - v^\ep)\A(\xi + \sigma(x) W) : \big(\nabla (\sigma(x) W) \otimes \nabla (\sigma(x) W)\big)\notag\\
 &\, + \delta(\xi - v^\ep) \A(\xi  +\sigma(x) W): \big(\nabla v^\ep \otimes \nabla (\sigma(x) W)\big). \label{eq:N_u}
\end{align}
The asymmetry in the cross term in failing to contain
both $\nabla v \otimes \nabla (\sigma(x) W)$ and $\nabla (\sigma(x) W)\otimes \nabla v$ arises from the fact
that the convex entropy used is $\Phi(v)$, instead of $\Phi(v + \sigma W)$.
One of the key insights in \cite{CP2003} is that, using the symmetry and nonnegativity of $\A$,
$\A$ can be written as the square of another symmetric,
positive semi-definite matrix so that (\ref{eq:N_u}) is non-negative.
The limit of $N^u_\ep$ is the non-negative parabolic defect measure $N^u$.

\smallskip
3. As before, we insert the regularizing operators: $\gamma(- \Delta)^\alpha  + \theta I$ (with fixed $\gamma$ and $\theta$ in this case) on both sides.
Again, we can decompose the solution into the following components:
\begin{align*}
\LL v(t), \varphi \RR = \LL v^0 + v^\flat + v^\sharp_F + v^\sharp_A + M_F + M_A + M_1 + M_2, \varphi \RR,
\end{align*}
with
\begin{align*}
&v^0(x,t) = \int \se(t) \chi^v(\xi,x,0)\;\d \xi ,\\
&v^\flat(x,t) = \int \int_0^t \se(s) (\gamma(-\Delta)^\alpha
+ \theta \mathrm{I}) \chi^v(\xi,x,t - s) \;\d s\,\d \xi ,\\
&v^\sharp_F(x,t) =  \int \int_0^t \se(t - s) \big(F'(\xi) - F'(\xi + \sigma(x) W)\big) \cdot \nabla \chi^v (\xi,x,s) \;\d s\,\d \xi ,\\
&v^\sharp_A(x,t) =  -\int \int_0^t \se(t - s) \nabla \cdot \big( (\A(\xi) - \A(\xi + \sigma(x) W))\nabla \chi^v (\xi, x,s)\big) \;\d s\,\d \xi ,\\
&\LL M_F, \varphi\RR = - \int \int_0^t  F'(v + \sigma(x) W) \cdot \nabla (\sigma(x) W)\, (\se^*(t - s) \varphi)(x,v(x,s))\;\d s\,\d x,\\
&\LL M_A, \varphi \RR= - \int \int_0^t   \A(v + \sigma(x)W) : \big(\nabla (\sigma(x) W) \otimes \nabla (\se^*(t - s) \varphi)(v(x,s),x)\big)\;\d s\,\d x,\\
&\LL M_1, \varphi\RR = - \iint \int_0^t \pd_\xi (\se^*(t - s) \varphi)\;\d (m^v + N^v) (\xi,x,s),\\
&\LL M_2, \varphi \RR = \int \int_0^t  \pd_\xi (\se^*(t - s) \varphi)(v(x,s),x)
  \A(v + \sigma(x) W):\big(\nabla (\sigma(x) W)\otimes \nabla (\sigma(x) W)\big)\d s \d x.
\end{align*}

Now we estimate each of these integrals, with some variations from \cite{DV2015} especially for the terms involving $\A$.
For this, $C>0$ is a universal constant, independent of $\eps, \tilde{\kappa}$, and $T$.

\medskip
4. We first have the familiar estimates:
\[
\int_0^T\|v^0(t)\|_{H^{\alpha}_x}^2 \;\d t \leq C \gamma^r\|u_0\|_{L^1_x},
\]
and
\[
\int_0^T\|v^\flat(t)\|_{L^2_x}^2 \;\d t \leq C \gamma^{r + 1} \int_0^T \|v(t)\|_{L^1_x}\;\d t
\]
from the velocity averaging arguments,
where $|r| < 1$ (we see that there is an extra power of $\gamma$ in the second estimate from those arguments, no matter what $r$ might be).

These imply
\begin{align}
&\fint_0^T \|v^0 \|_{L^1_x} \;\d t  \leq  C T^{-\frac{1}{2}}\gamma^{\frac{r}{2}} \|u_0\|_{L^1_x}^{\frac{1}{2}}, \label{eq:smallball1}\\
&\fint_0^T\|v^\flat(t)\|_{L^1_x} \;\d t \leq  C \gamma^{\frac{r + 1}{2}} \Big(\fint_0^T \|v(t)\|_{L^1_x}\;\d t \Big)^{\frac{1}{2}}.
 \label{eq:smallball2}
\end{align}

\smallskip
5. For $v^\sharp_F$ and $v^\sharp_A$, we use the fact that
\begin{align*}
&\big(F'(\xi) - F'(\xi + \sigma(x) W)\big)\cdot \nabla\chi^v(\xi,x,s)\\
&\quad =\nabla \cdot \big((F'(\xi) - F'(\xi + \sigma(x) W) ) \chi^v(\xi,x,s)\big)
  - \big(F''(\xi + \sigma(x) W) \cdot \nabla \sigma(x)W\big)\chi^v(\xi,x,s), \\
&\big(\A(\xi) - \A(\xi + \sigma(x) W)\big) \nabla \chi^v(\xi,x,s) \\
&\quad = \nabla \cdot \big((\A(\xi) - \A(\xi + \sigma(x) W) ) \chi^v(\xi,x,s)\big)
 - \big(\A'(\xi + \sigma(x) W) \nabla \sigma(x)W \big)\chi^v(\xi,x,s).
\end{align*}

Now we apply the kernel estimates.
Let $\varphi \in L^2$ be any test function, and let $\LL\cdot, \cdot \RR$ be the pairing in $L^2$. Then
\begin{align}
\LL v^\sharp_F(t), \varphi \RR  =  &  \iint \int_0^t \varphi \se(t - s) \nabla \cdot\big((F'(\xi) - F'(\xi + \sigma(x) W))\, \chi^v(\xi,x,s)\big) \;\d s\,\d \xi \,\d x\notag \\
& -   \iint \int_0^t \varphi\, \se(t - s) \big((F''(\xi + \sigma(x) W)\cdot\nabla \sigma(x) W)\,\chi^v(\xi,x,s)\big) \;\d s\,\d \xi \,\d x \notag\\
= &  \iint \int_0^t\nabla(\se^*(t - s)\varphi)  \cdot\big(F'(\xi) - F'(\xi + \sigma(x) W)\big)\, \chi^v(\xi,x,s)\;\d s\,\d \xi \,\d x \notag\\
& -   \iint \int_0^t \se^*(t - s) \varphi\,\big(F''(\xi + \sigma(x) W)\cdot \nabla \sigma(x) W\big) \chi^v(\xi,x,s)\;\d s\,\d \xi \,\d x. \label{5.7a}
\end{align}
Similarly, we have
\begin{align}
\LL v^\sharp_A(t), \varphi \RR
= & \iint \int_0^t  \varphi \se(t - s) \nabla^2 : \big((\A(\xi) - \A(\xi + \sigma(x) W)) \chi^v(\xi,x,s)\big)  \;\d s\,\d \xi \,\d x\notag\\
 & - \iint \int_0^t  \varphi \se(t - s) \nabla \cdot \big(\A'(\xi + \sigma(x) W) \nabla (\sigma(x)W) \chi^v(\xi,x,s)\big) \;\d s\,\d \xi \,\d x\notag\\
=& \iint \int_0^t   \nabla^2(\se^*(t - s)\varphi)  : \big(\A(\xi) - \A(\xi + \sigma(x) W)\big) \chi^v(\xi,x,s) \;\d s\,\d \xi \,\d x\notag\\
& - \iint \int_0^t  \big(\nabla (\se^*(t - s)\varphi) \otimes \nabla \sigma(x)  W\big): \A'(\xi + \sigma(x) W)\, \chi^v(\xi,x,s)\;\d s\,\d \xi \,\d x. \label{5.7b}
\end{align}

Notice that
\begin{align}
&\int_0^T \int_0^t\|\nabla(\se^*(t - s)\varphi) \cdot\big(F'(\cdot) - F'(\cdot + \sigma(\cdot) W)\big)\, \chi^v(\cdot,\cdot,s) \|_{L^1_{x,\xi}} \;\d s\,\d t \notag\\
&\leq \int_0^T \int_0^t\|\nabla(\se^*(t - s)\varphi) \|_{L^\infty_{x,\xi}} \|F'(\cdot) - F'(\cdot + \sigma(\cdot) W)\|_{L^\infty_{x,\xi}} \|\chi^v(\cdot,\cdot,s) \|_{L^1_{x,\xi}}\;\d s\,\d t \notag\\
&\leq \int_0^T \int_0^t \|\nabla\se^*(t - s) \|_{L^2 \to L^\infty} \|\varphi\|_{L^2_x} \|F'(\cdot) - F'(\cdot + \sigma(\cdot) W)\|_{L^\infty_{x,\xi}} \|\chi^v(\cdot,\cdot,s) \|_{L^1_{x,\xi}}\;\d s\,\d t \notag\\
&\leq  C\tilde{\kappa}\|\varphi\|_{L^2_x}\sup_{s \in [0,T]} \Big(\int_0^T  e^{\theta (t - s)} (\gamma t)^{-\frac{d + 2}{4\alpha}} \;\d t\Big)
    \int_0^T   \|v(s)\|_{L^1_x}\;\d s\notag\\
&\leq   C\tilde{\kappa}\|\varphi\|_{L^2_x}\gamma^{-\frac{d + 2}{4\alpha}} \theta^{\frac{d + 2}{4 \alpha} - 1}\int_0^\infty  e^{-t}  t^{-\frac{d + 2}{4\alpha}} \;\d t \;\int_0^T \|v(s)\|_{L^1_x}\;\d s\notag \\
&=  C\tilde{\kappa} \|\varphi\|_{L^2_x} \gamma^{-\frac{d + 2}{4\alpha}} \theta^{\frac{d + 2}{4 \alpha} - 1} \big|\Gamma(1 - \frac{d + 2}{4 \alpha})\big|\;\int_0^T \|v(s)\|_{L^1_x}\;\d s; \label{5.7a-1}
\end{align}

\smallskip
\begin{align}
&\int_0^T \int_0^t\|\big(\se^*(t - s)\varphi\big)\, \nabla (\sigma(x) W)\cdot\big(F''(\cdot + \sigma(\cdot) W) \chi^v(\cdot,\cdot,s)\big) \|_{L^1_{x,\xi}}\;\d s\,\d t \notag\\
&\leq \int_0^T \int_0^t\|\se^*(t - s)\varphi \|_{L^\infty_{x,\xi}} \|F''(\cdot + \sigma(\cdot) W)\|_{L^\infty_{x,\xi}} \|\sigma W\|_{W^{1,\infty}_x} \|\chi^v(\cdot,\cdot,s)\|_{L^1_{x,\xi}}\;\d s\,\d t \notag\\
&\leq C\tilde{\kappa}\|\varphi\|_{L^2_x}\gamma^{-\frac{d}{4\alpha}} \theta^{\frac{d}{4\alpha} -1} \big|\Gamma(1 -\frac{d}{4\alpha})\big|\;\int_0^T \|v(s)\|_{L^1_x}\;\d s;\label{5.7a-2}
\end{align}

\smallskip
\begin{align}
&\int_0^T \int_0^t\|\nabla^2(\se^*(t - s)\varphi):\big(\A(\cdot) - \A(\cdot + \sigma(\cdot) W)\big)\, \chi^v(\cdot,\cdot,s)\|_{L^1_{x,\xi}}\;\d s\,\d t\notag\\
&\leq \int_0^T \int_0^t\|\nabla^2(\se^*(t - s)\varphi) \|_{L^\infty_{x,\xi}} \|\A(\cdot) - \A(\cdot + \sigma(\cdot) W)\|_{L^\infty_{x,\xi}} \|\chi^v(\cdot,\cdot,s)\|_{L^1_{x,\xi}}\;\d s\,\d t\notag\\
&\leq C \tilde{\kappa} \|\varphi\|_{L^2_x}\gamma^{-\frac{d + 4}{4\alpha}} \theta^{\frac{d + 4}{4 \alpha} - 1}
  \big|\Gamma(1 - \frac{d + 4}{4 \alpha})\big| \;\int_0^T  \|v(s)\|_{L^1_x}\;\d s;\label{5.7b-1}
\end{align}

\smallskip
\begin{align}
&\int_0^T \int_0^t\|\big(\nabla(\se^*(t - s)\varphi) \otimes \nabla (\sigma(\cdot) W)\big) :\A'(\cdot + \sigma(\cdot) W)\, \chi^v(\cdot,\cdot,s) \|_{L^1_{x,\xi}}\;\d s\,\d t\notag\\
&\leq \int_0^T \int_0^t\|\nabla(\se^*(t - s)\varphi) \|_{L^\infty_{x,\xi}} \|\sigma W\|_{W^{1,\infty}_x}
\|\A'(\cdot + \sigma(\cdot) W)\|_{L^\infty_{x,\xi}} \|\chi^v(\cdot,\cdot,s) \|_{L^1_{x,\xi}}\;\d s\,\d t\notag\\
&\leq  C \tilde{\kappa} \|\varphi\|_{L^2_x}\gamma^{-\frac{d + 2}{4\alpha}} \theta^{\frac{d + 2}{4\alpha} - 1}
 \big|\Gamma(1 - \frac{d + 2}{4\alpha})\big|  \;\int_0^T \|v(s)\|_{L^1_x}\;\d s.\label{5.7b-2}
\end{align}

Now, by (\ref{eq:conditions_uniqueness}),  we have assumed that
\begin{align*}
|F''(\xi)| \lesssim  1,\qquad
|\A'(\xi)| \lesssim  1,
\end{align*}
and $\|\sigma(x) W\|_{W^{1,\infty}} \leq \tilde{\kappa}$, so that we can use the estimates
(the second from the first by the Poincar\'e-Wirtinger inequality, since $\int_{\mathbb{T}^d} \sigma(x) \;\d x = 0$):
\begin{align*}
&|F'(\xi) - F'(\xi + \sigma(x) W)|
+|F''(\xi + \sigma(x) W) \cdot \nabla (\sigma W)|\leq C \tilde{\kappa}, \\
&|\A(\xi) - \A(\xi + \sigma(x) W) |+
|\A'(\xi + \sigma(x) W) \nabla (\sigma W) |\leq C \tilde{\kappa}.
\end{align*}

Putting these estimate \eqref{5.7a}--\eqref{5.7b-2} back into
the bound: $\|v^\sharp_\cdot(t)\|_{L^2_x} = \sup_{\|\varphi\|_{L^2_x} = 1} \LL v^\sharp_\cdot(t) , \varphi\RR$,
we have
\begin{align}
&\int_0^T \|v^\sharp_A(t) + v^\sharp_F(t)\|_{L^2_x} \;\d t  \notag  \\
&\leq  C\tilde{\kappa} \Big(\gamma^{-\frac{d + 2}{4\alpha}} \theta^{\frac{d + 2}{4 \alpha} - 1}|\Gamma(1 - \frac{d + 2}{4\alpha})|
   + \gamma^{-\frac{d}{4\alpha}} \theta^{\frac{d}{4\alpha} -1} |\Gamma(1 -\frac{d}{4\alpha} )|\notag\\
&\quad \qquad + \gamma^{-\frac{d + 4}{4\alpha}} \theta^{\frac{d + 4}{4 \alpha} - 1}|\Gamma(1 - \frac{d + 4}{4\alpha})|\Big)
 \int_0^T \|v(t)\|_{L^1_x}\;\d t . \label{eq:smallball3}
\end{align}

\smallskip
6. For $M_F$ and $M_A$, we employ the kernel estimate and $\|\sigma W\|_{W^{1,\infty}_x} \leq \tilde{\kappa}$ to obtain
\begin{align*}
&|\LL  M_F, \varphi \RR| \leq \int_0^t\|F'(v + \sigma W)\|_{L^1_x} \| \se \varphi\|_{L^\infty_x}\|\nabla \sigma W\|_{L^\infty_x}\;\d s,\\
& |\LL  M_A, \varphi\RR|\leq  \int_0^t \|\A(v + \sigma W)\|_{L^1_x} \| \nabla (\se \varphi)\|_{L^\infty_x}\|\nabla \sigma W\|_{L^\infty_x}\;\d s.
\end{align*}

Now, by (\ref{eq:conditions_existence}), we have
\begin{align*}
\|F'(v + \sigma W)\|_{L^1_x}+\|\A(v + \sigma W)\|_{L^1_x}  \leq C \big(1 + \|v(t)\|_{L^1_x} + \|\sigma \|_{L^1_x}|W|\big).
\end{align*}

These give
\begin{align}
&\int_0^T \|M_F(t) + M_A(t)\|_{L^1_x} \;\d t  \notag \\
&\leq   C \tilde{\kappa}\int_0^T \int_0^t \big(1 + \|v(s)\|_{L^1_x}\big) e^{-\theta(t - s)}\big(1 + (\gamma (t - s))^{ -\frac{1}{2\alpha}}\big) \;\d s \;\d t \notag\\
&\leq  C\tilde{\kappa}  \Big(\theta^{-1} + \gamma^{ -\frac{1}{2\alpha}} \theta^{\frac{1}{2\alpha} - 1} |\Gamma(1 - \frac{1}{2\alpha})|\Big)
  \int_0^T \big(1 + \|v(s)\|_{L^1_x}\big) \;\d s.
  \label{eq:smallball4}
\end{align}

\smallskip
7. For $M_2$, we have
\begin{align*}
\LL  M_2, \varphi \RR
=\int \int_0^t  \pd_\xi (\se^*(t - s) \varphi)(v(x,s),x) \A(v + \sigma(x) W) :\big(\nabla (\sigma(x) W)\otimes \nabla(\sigma(x) W)\big)\,\d s\,\d x.
\end{align*}
We notice that
\begin{align*}
\pd_\xi (\se(t-s)\varphi) (v(x,s),x)= (t - s) F''(v(x,s)) \cdot \nabla (\se^*(t - s)\varphi) + \A'(v(x,s)):\nabla^2 (\se^*(t - s) \varphi),
\end{align*}
as explained in (\ref{eq:pd_xi_se}).
By (\ref{eq:conditions_uniqueness}), we have assumed that
\begin{align*}
|F''(\xi)| \lesssim  1,\qquad
|\A'(\xi)| \lesssim  1.
\end{align*}

Again we have
\begin{align*}
\|\A(v + \sigma(x) W) \|_{L^1_x} \leq  C\big(1 + \|v(s)\|_{L^1_x} + \|\sigma\|_{L^1_x}|W|\big).
\end{align*}

Finally, using the kernel estimate yields
\begin{align*}
&|\LL M_2,\varphi\RR |\\
&\leq  C\int_0^t (t - s)\|F''\|_{L^\infty}\|\nabla S^*(t - s) \varphi\|_{L^\infty_{x,\xi}} \|\nabla \sigma W\|_{L^\infty_x}^2
   \big(1 + \|v(s)\|_{L^1_x} + \|\sigma\|_{L^1_x}|W|\big) \;\d s\\
&\quad + C\int_0^t (t - s)\|\A'\|_{L^\infty}\|\nabla^2 S^*(t - s) \varphi\|_{L^\infty_{x,\xi}} \|\nabla \sigma W\|_{L^\infty_x}^2
     \big(1 + \|v(s)\|_{L^1_x} + \|\sigma\|_{L^1_x}|W|\big) \;\d s\\
&\leq  C\tilde{\kappa}^2\int_0^t (t - s)\big(\|\nabla S^*(t - s) \|+\|\nabla^2 S^*(t - s)\|\big)
\|\varphi\|_{L^\infty_x} \big(1 + \|v(s)\|_{L^1_x} + \|\sigma\|_{L^1_x}|W|\big)\;\d s.
\end{align*}
Therefore, we have
\begin{align*}
&\int_0^T \|M_2(t)\|_{L^1_x} \;\d t  \\
&\leq   C\tilde{\kappa}^2 \int_0^T\int_0^t (t - s)
  \big(\|\nabla S^*(t - s) \|+\|\nabla^2 S^*(t - s)\|\big)
\big(1 + \|v(s)\|_{L^1_x} + \|\sigma\|_{L^1_x}|W|\big) \;\d s\,\d t \\
&\leq  C\tilde{\kappa}^2 \int_0^T\int_0^t (t - s)\big((\gamma (t - s))^{-\frac{1}{2\alpha}}+(\gamma (t - s))^{-\frac{1}{\alpha}}\big)
   e^{-\theta (t - s)}\big(1 + \|v(s)\|_{L^1_x} + \|\sigma\|_{L^1_x}|W|\big)\,\d s\, \d t,
\end{align*}
and
\begin{align}
&\int_0^T \|M_2(t)\|_{L^1_x} \;\d t \label{eq:smallball5}\\
&\leq  C \tilde{\kappa}^2 \Big(\gamma^{-\frac{1}{2\alpha}} \theta^{\frac{1}{2\alpha} - 2} \big|\Gamma(2 - \frac{1}{2\alpha})\big|
 + \gamma^{-\frac{1}{\alpha}} \theta^{\frac{1}{\alpha} - 2} \big|\Gamma(2 - \frac{1}{\alpha})\big|\Big)\int_0^T \big(1 + \|v(s)\|_{L^1_x}\big) \;\d s.\notag
\end{align}

\smallskip
8. For the kinetic measure $M_1$, we use the total variation estimate again.
First, with $\varphi \in L^\infty_x$,
\begin{align*}
&|\LL M_1, \varphi\RR| \\
&= \left| \iint \int_0^t \pd_\xi (\se^*(t - s) \varphi)\; \d(m^v + N^v)(x,\xi,s) \right|\\
&= \left|\iint \int_0^t \big((t - s) F''(\xi) \cdot \nabla (\se^*(t - s)\varphi) + \A'(\xi) : \nabla^2 (\se^*(t - s) \varphi)\big)\; \d(m^v + N^v)(x,\xi,s)\right|\\
&\leq  C\|\varphi\|_{L^\infty_x} \iint \int_0^t \big(\gamma^{-\frac{1}{2\alpha}}(t - s)^{1 - \frac{1}{2\alpha}}
  +\gamma^{-\frac{1}{\alpha}}(t - s)^{1 - \frac{1}{\alpha}}\big)e^{-\theta (t - s)} \d|m^v + N^v|(x,\xi,s),
\end{align*}
so that
\begin{align*}
&\int_0^T \| M_1(t)\|_{L^1_x}\;\d t \\
&\leq  C\Big(\gamma^{-\frac{1}{2\alpha}} \theta^{\frac{1}{2\alpha} - 2} \big|\Gamma(2 - \frac{1}{2\alpha})\big|
       +\gamma^{-\frac{1}{\alpha}} \theta^{\frac{1}{\alpha} - 2} \big|\Gamma(2 - \frac{1}{\alpha})\big|\Big)|m^v + N^v|(\R\times \mathbb{T}^d\times [0,T]).
\end{align*}

As in Lemma \ref{thm:TVestimate1}, we test equation (\ref{eq:kinetic_equation2}) against $\xi$ to find
\begin{align*}
&\frac{1}{2}\|v(t)\|_{L^2_x}^2  + |m^v + N^v|(\mathbb{R}\times \mathbb{T}^d\times[0,t])\\
&\leq  \frac{1}{2} \|\tilde{u}_0\|_{L^2_x}^2 + \left|\int_0^t \iint \xi \big(F'(\xi)  - F'(\xi + \sigma(x) W)\big) \cdot \nabla \chi^v \;\d \xi \,\d x\,\d s\right|\\
& \quad + \left|\int_0^t \int v F'(v + \sigma(x) W) \cdot \nabla (\sigma(x) W) \;\d x\,\d s\right|\\
& \quad + \left|\int_0^t \int \A(v + \sigma(x) W) : \big(\nabla(\sigma W) \otimes \nabla (\sigma W)\big) \;\d x\,\d s\right|\\
&\leq  \frac{1}{2} \|\tilde{u}_0\|_{L^2_x}^2 + \left|\int_0^t \iint \xi F''(\xi + \sigma(x) W)) \cdot \nabla(\sigma(x) W) \chi^v \;\d \xi \,\d x\,\d s\right|\\
&\quad  + C\tilde{\kappa}(1+\tilde{\kappa}) \left|\int_0^t \int v \big(1 + |v| + |\sigma(x) W|\big) \;\d x\,\d s\right|\\
&\leq  \frac{1}{2} \|\tilde{u}_0\|_{L^2_x}^2 + C\tilde{\kappa} \int_0^t \big(1 + \|v(s)\|_{L^2_x}^2\big)\;\d s.
\end{align*}

Then Gronwall's inequality implies
\begin{align*}
|m^v + N^v|\leq C e^{C\tilde{\kappa} t} \big(\|\tilde{u}_0\|_{L^2_x}^2 + 1\big) \leq C e^{C\tilde{\kappa} t} \big( \hat{\kappa}^2 \eps^{-d} + 1\big).
\end{align*}
Therefore, we have
\begin{align}
&\int_0^T \| M_1(t)\|_{L^1_x}\;\d t \label{eq:smallball6}\\
&\leq  C\Big(\gamma^{-\frac{1}{2\alpha}} \theta^{\frac{1}{2\alpha} - 2} \big|\Gamma(2 - \frac{1}{2\alpha})\big|
+ \gamma^{-\frac{1}{\alpha}} \theta^{\frac{1}{\alpha} - 2} \big|\Gamma(2 - \frac{1}{\alpha})\big|\Big) e^{C\tilde{\kappa} T} \big( \hat{\kappa}^2 \eps^{-d} + 1\big).
\notag
\end{align}

\smallskip
9. \textit{Completion of the estimates}. First, we set $\alpha \leq \frac{1}{2}$ so that the instances of $|\Gamma|$ are never evaluated at a negative integer,
where it is infinite.
With the finite bound of all the values of $|\Gamma|$ and finitely many instances of $\Gamma$ in estimates (\ref{eq:smallball1})--(\ref{eq:smallball6}) above,
we can write those estimates as
\begin{align*}
&\fint_0^T \|v^0(t)\|_{L^1_x}\;\d t  \leq CT^{-\frac{1}{2}}\gamma^{\frac{r}{2}} \|u_0\|_{L^1_x}^{\frac{1}{2}},\\
&\fint_0^T \|v^\flat(t)\|_{L^1_x}\;\d t  \leq  C\gamma^{\frac{r + 1}{2}}
    \Big(\fint_0^T \|v\|_{L^1_x}\;\d t \Big)^{\frac{1}{2}}
     \leq  C\gamma^{\frac{r + 1}{2}}  \fint_0^T \big(1 + \|v\|_{L^1_x}\big)\;\d t ,\\
&\fint_0^T \|v^\sharp_F + v^\sharp_A\|_{L^1_x}\;\d t
   \leq C\tilde{\kappa} \Big(\gamma^{-\frac{d + 2}{4\alpha}} \theta^{\frac{d + 2}{4\alpha} - 1} + \gamma^{-\frac{d}{4\alpha}} \theta^{\frac{d}{4\alpha} - 1}
      + \gamma^{-\frac{d + 4}{4\alpha}} \theta^{\frac{d + 4}{4\alpha} - 1}\Big) \fint_0^T \|v(t)\|_{L^1_x} \;\d t ,\\
&\fint_0^T \|M_F(t) + M_A(t)\|_{L^1_x}\;\d t  \leq  C\tilde{\kappa} \left(\theta^{-1} + \gamma^{-\frac{1}{2\alpha}} \theta^{\frac{1}{2\alpha} - 1}\right)
    \fint_0^T \big(1 + \|v(t)\|_{L^1_x} \big)\;\d t ,\\
&\fint_0^T \|M_2(t)\|_{L^1_x} \;\d t \leq  C\tilde{\kappa}^2 \left( \gamma^{-\frac{1}{2\alpha}} \theta^{\frac{1}{2\alpha} - 2}
    + \gamma^{-\frac{1}{\alpha}} \theta^{\frac{1}{\alpha} - 2}\right) \fint_0^T\big(1 + \|v(t)\|_{L^1_x}\big)\;\d t ,\\
&\fint_0^T \|M_1\|_{L^1_x} \;\d t  \leq \frac{C}{T} \left( \gamma^{-\frac{1}{2\alpha}} \theta^{\frac{1}{2\alpha} - 2}
+  \gamma^{-\frac{1}{\alpha}} \theta^{\frac{1}{\alpha} - 2}\right) e^{C_0\tilde{\kappa} T}\big(\hat{\kappa}^2 \eps^{-d} + 1\big).
\end{align*}

Combining all the estimates together yields
\begin{align*}
\fint_0^T \|v(t)\|_{L^1_x} \;\d t
\leq&\,  C_1T^{-\frac{1}{2}}\gamma^{\frac{r}{2}} \|u_0\|_{L^1_x}^{\frac{1}{2}}
+ \Big(C_2\gamma^{\frac{r + 1}{2}}
    + C_3(\gamma,\theta) (\tilde{\kappa} + \tilde{\kappa}^2)\Big) \Big(1 + \fint_0^T \|v\|_{L^1_x}\;\d s \Big)\\
  &\, + C_4(\gamma,\theta) T^{-1}e^{c\tilde{\kappa} T} (\hat{\kappa}^2 \eps^{-d} + 1).
\end{align*}

We can choose $\gamma$, $\theta$, $T$, and $\tilde{\kappa}$ in that order so that, for some $q$ to be determined,
\begin{align*}
C_2 \gamma^{\frac{r +1}{2}} \leq q\eps.
\end{align*}
For $\alpha < \frac{1}{4}$,
we see that every $\theta$ has positive power above, except in $C_3(\gamma,\theta)$ for the estimate of $\|M_F + M_A\|_{L^1_{x,t}}$,
so that $C(\rho,\theta)$ involves $\theta$ with positive power.
Therefore, we choose $\theta$ such that
\[
C_4(\gamma,\theta) < 1,
\]
so that we can choose $T$ sufficiently large such that
\[
C_1T^{-\frac{1}{2}} \|u_0\|_{L^1_x}^{\frac{1}{2}}
  + C_4(\gamma,\theta)T^{-1}\big(\hat{\kappa}^2\eps^{-d} + 1\big)
  \leq q\eps.
\]
Finally, we choose $\tilde{\kappa}$ such that
\[
C_3(\gamma, \theta) \tilde{\kappa}(1 + \tilde{\kappa}) \leq q\eps,
\]
and
\[
C_0 \tilde{\kappa} T \leq q \eps.
\]

By taking $q$ sufficiently small, we have
\begin{align*}
\fint_0^T \|v(t)\|_{L^1_x} \d t  \leq \frac{\eps}{4}, \qquad \fint_0^T \|\tilde{u}(t) \|_{L^1_x} \;\d t  \leq \frac{3\eps}{8},
\end{align*}
which leads to
\begin{align*}
\fint_0^T \|u(t)\|_{L^1_x} \;\d t  \leq \frac{\eps}{2},
\end{align*}
by the following $L^1$-contraction property of the pathwise solutions:

\begin{lma}
Let $\A$ be symmetric positive-semi-definite, and let both $\A(\xi)$ and $F(\xi)$ be H\"older continuous
and of polynomial growth.
Then, for each initial data function $u_0$, there exists a unique measurable $u: \mathbb{T}^d\times [0,T] \times \Omega \to \mathbb{R}$
solving  \eqref{3.5a} in the sense of Definition {\rm \ref{def:ksolution}}.
Moreover,
for $u_0, \tilde{u}_0 \in L^1(\mathbb{T}^d)$,
$$
\|u(t)-\tilde{u}(t)\|_{L^1_x}\le \|u_0-\tilde{u}_0\|_{L^1_x} \qquad \mbox{almost surely}.
$$
\end{lma}

\medskip
For our case,
$$
\|u(t)-\tilde{u}(t)\|_{L^1_x}\le \|u_0-\tilde{u}_0\|_{L^1_x}\le \frac{\eps}{8} \qquad \mbox{almost surely}.
$$
This completes the proof.
\end{proof}

\begin{rem}
This almost sure $L^1$--contraction property is not available in the multiplicative case.
In fact, it is not available in many other situations, such as in systems or for non-conservative equations
where the $L^1$-contraction is not ready to provide a stability condition.
It is of interest to study the uniqueness and ergodicity properties of invariant measures for equations without this property.
\end{rem}

\subsection{Uniqueness III: Conclusion} \label{sec:uniquenessIII}

Let $u^1_0$ and $u^2_0$ be in $L^1_x$.
For a given $\eps > 0$, let $\tilde{u}^1_0$ and $\tilde{u}^2_0$ be in $L^3_x$
such that
$$
\|u^i_0 - \tilde{u}^i_0\|_{L^1_x} \leq \frac{\eps}{4}.
$$
Denote their corresponding solutions $u^1,u^2,\tilde{u}^1$, and $\tilde{u}^2$, respectively.
Let us now put $\tilde{u}^1$ and $\tilde{u}^2$, in place of $u$ and $v$ in \S \ref{sec:uniquenessI},
and the corresponding sequence of stopping times constructed recursively in (\ref{eq:stopping_time}):
\begin{align*}
\mathcal{T}_l = \inf\{ t \geq \mathcal{T}_{l - 1} + T : \|\tilde{u}^1(t) \|_{L^1_x} + \|\tilde{u}^2(t)\|_{L^1_x} \leq 2 \hat{\kappa}\}.
\end{align*}

As in \cite{DV2015}, choosing $T$ and $\tilde{\kappa}$ as above, we obtain by the $L^1$--contraction (for the additive noise, there is the $L^1$-contraction almost sure):
\begin{align*}
&\mathbb{P}\big\{\fint_{\mathcal{T}_l}^{\mathcal{T}_l + T} \|u^1(s) - u^2(s)\|_{L^1_x }\;\d s \leq \eps\,\, \big| \,\mathscr{F}_{\mathcal{T}_l}\big\}\\
&\geq  \mathbb{P} \big\{\fint_{\mathcal{T}_l}^{\mathcal{T}_l + T} \|\tilde{u}^1(s) - \tilde{u}^2(s)\|_{L^1_x} \;\d s \leq \frac{\eps}{2}\,\, \big|\,\mathscr{F}_{\mathcal{T}_l}\big\}\\
&\geq  \mathbb{P}\big\{\sup_{t \in [\mathcal{T}_l, \mathcal{T}_l + T]} \|\sigma W(t) - \sigma W(\mathcal{T}_l)\|_{W^{1,\infty}_x} \leq \tilde{\kappa}\,\, \big|\,\mathscr{F}_{\mathcal{T}_l}\big\}.
\end{align*}

Since $\tilde{\kappa} > 0$, and $\sigma$ is Lipschitz,
we can denote the positive probability of the event as $\lambda$.
By the strong Markov property, we know that it does not change with $l$.

This allows us to write
\begin{align*}
&\mathbb{P}\big\{\fint_{\mathcal{T}_l}^{\mathcal{T}_l + T} \|u^1(s) - u^2(s)\|_{L^1_x }\;\d s \geq \eps \,\,\, \mbox{ for } l = l_0, l_0 + 1, \ldots, l_0 + k \big\}
\leq (1 - \lambda)^k,
\end{align*}
so that
\begin{align*}
&\mathbb{P}\big\{\lim_{l \to \infty} \fint_{\mathcal{T}_l}^{\mathcal{T}_l + T} \|u^1(s) - u^2(s)\|_{L^1_x}\;\d s \geq \eps\big\}\\
&=  \mathbb{P}\big\{\exists l_0 \,\,\forall l \geq l_0 \,:\,  \fint_{\mathcal{T}_l}^{\mathcal{T}_l + T} \|u^1(s) - u^2(s)\|_{L^1_x}\;\d s \geq \eps\big\}\\
&=  0.
\end{align*}
This limit exists as $s \mapsto \|u^1(s) - u^2(s)\|_{L^1_x}$ is non-increasing, by the $L^1$--contraction property.
Then, by the same property,
\[
\mathbb{P}\big\{\lim_{t \to \infty} \|u^1(t) - u^2(t) \|_{L^1_x} \geq \eps\big\} = 0.
\]
Therefore, almost surely,
\[
\lim_{t \to \infty} \|u^1(t) - u^2(t) \|_{L^1_x}  = 0,
\]
which implies the uniqueness of the invariant measure.

\section{Further Developments, Problems, and Challenges}

In this section, we discuss some further developments, problems, and challenges in this direction.

\subsection{Further problems}
There are several natural problems that follow from the analysis discussed above.
We restrict ourselves again to nonlinear conservation laws driven by stochastic forcing.

One of the problems is the long-time behavior problem for solutions of nonlinear conservation laws driven by multiplicative noises.
The noises of form $\nabla \cdot (F(u) \circ \d W)$ have been considered in \cite{GS2014,GS2016},
in which the dynamics remains in the zero-spatial-average subspace of $L^1(\mathbb{T}^d)$.

The well-posedness for nonlinear conservation laws driven by multiplicative noises
is quite well understood
from several different perspectives -- the strong entropy stochastic solutions of Feng-Nualart \cite{FN2008}
and of Chen-Ding-Karlsen \cite{CDK2012},
the viscosity solution methods of Bauzet-Vallet-Wittbold\cite{BVW2012},
and the kinetic approach of Debussche-Hofmanov\`a-Vovelle \cite{DHV2014,DV2010},
as we have mentioned above.
Nevertheless, the long-time behavior problem for solutions is wide open,
since there is no effective way to control $\|u(t)\|_{L^1_x}$.

We remark on two aspects of the noises that can affect qualitative long-time behaviors of solutions:
\begin{itemize}
\item[(i)] The question seems to depend heavily on the roots and growth of the noise coefficient
function $\sigma(u)$ -- If the noise is
  degenerate (not cylindrical), say $\sigma(u) = 0$ for certain $u=r \in \mathbb{R}$,
  then $u \equiv r$ is a fixed point of the evolution.
  By the $L^1$--contraction, it is possible to prove certain long-time behavior results
  for the solutions for the unbounded noise coefficient function with one root.
  Both the growth of $\sigma$ and how many roots it possesses affect the long-time behaviors of solutions,
  as is evident also in the analysis of other equations such as the KPP equation (discussed below).
  In the case that $\sigma$ has no roots, there are no fixed points.
  It is possible that the nonlinear conservation laws driven by bounded noises with no roots
  have non-trivial invariant measures.

\item[(ii)]  If the noise is $\sigma(u) \d B = \sum_k g_k(u) \d W^k$, where $B = W^k e_k $ is a cylindrical Wiener process,
the behaviors are expected to be very different from the case that the noise is simply $\sigma(u) \d W$.
\end{itemize}

Another natural direction to consider is the case of nonlinear systems of balance laws.
For this case, such as for the isentropic Euler system, the kinetic formulation is not ``pure" -- it
contains the instances of the solution mixed with the kinetic operator ({\it cf.} \cite{LPT1994b}).
At present, it seems that the methods discussed above
are not directly applicable to the systems.

\subsection{The Navier-Stokes equations}\label{sec:NSE}

The two-dimensional incompressible Navier-Stokes equations (INSEs) driven by stochastic forcing has been a subject of intense interest.
We focus on the analysis of asymptotic behaviors of solutions to keep ourselves from getting sidetracked.

The existence of invariant measures for the 2-D INSEs on a regular bounded domain with a general noise
that is a Gaussian random field and white-in-time has been known at least since \cite{Fla1994}.
The uniqueness and ergodicity for the 2-D INSEs have also been established;  see \cite{FM1995} and the references therein
for such results and further existence results of invariant measures under different conditions.
These results have subsequently been improved, including for the noises that are localized in time
and Gaussian in space, in \cite{BKL2001,BKL2002,MY2002,Mat2002}, and in some references
cited in this paper.

We remark particularly that the corresponding existence  questions
for the 2-D INSEs with multiplicative noises have been established, for example in \cite{FG1995},
via the Skorohod embedding and a Faedo-Galerkin procedure,
which have shown the existence of martingale solutions and stationary martingale solutions,
from which in turn the existence of an invariant measure can be derived.

The asymptotic behaviors of 2-D INSEs driven by white-in-time noises
or Poisson distributed unbounded kick noises have been explored,
and the existence and uniqueness of invariant measures for these systems are known.
See also \cite{Mat2003,KS2006,KS2012} for the related references.

There are also more recent results on INSEs driven by space-time white noises
in 2-D or 3-D; see \cite{DD2002,DD2003,ZZ2015} and the references therein.
For example, it is known that the transition semigroup of the Kolmogorov equation
associated to the 3-D stochastic INSEs driven by a cylindrical
white noise has a unique (and hence ergodic) invariant measure.

The existence of invariant measures for the compressible Navier-Stokes equations, even in the 2-D case, is wide open.

See \cite{EMS2001,MT2016}, and discussions in \cite[Chp. 3]{Madja_2016}
as well as references contained there for further treatments on ergodicity results; also see \cite{MadjaWang}.

\subsection{The asymptotic strong Feller property}
Using the 2-D INSEs as a springboard,
the notion of the {\it asymptotic strong Feller} property has been introduced in Hairer-Mattingley \cite{HM2006}
as a weaker and more natural replacement of the sufficient ``strong Feller" property
(Definition \ref{defin:strongfeller}) in dissipative infinite-dimensional systems,
the possession of which guarantees the uniqueness of an invariant measure.
In the finite-dimensional case of SDEs, there is a related notion of {eventual} strong Feller property,
for which sufficient conditions are given in \cite{Bie2011}.

The definition of asymptotic strong Feller property depends on a preliminary definition:

\begin{defin}[Totally separating system]
A pseudo-metric is a function $d : \mX^2 \to \mathbb{R}^+_0$ for which $d(x,x) = 0$
and the triangle inequality is satisfied,
and $d_1 \geq d_2$ if the inequality holds for all arguments $(x,y) \in \mX^2$.
Let $\{d_k\}_{k = 0}^\infty$ be an increasing sequence of pseudo-metrics
on a Polish space $\mX$.
Then $\{d_k\}_{k = 0}^\infty$ is a {\it totally separating system} of pseudo-metrics if
\[
\lim_{k \to 0} d_k = 1
\]
pointwise everywhere off the diagonal on $\mX^2$.
\end{defin}

Then the asymptotic strong Feller property is defined as follows:

\begin{defin}[Asymptotic Strong Feller property]
A Markov transition semigroup $\mathscr{P}_t$ on a Polish space $X$ is {\it asymptotically strong Feller}
at $x$ if there exist both a totally separating system of pseudo-metrics $\{d_k\}_{k = 0}^\infty$
and an increasing sequence of times $t_k$ such that
$$
\inf_{U \in \nb(x)} \limsup_{k \to \infty} \sup_{y \in U}  \|P_{t_k}(x,\cdot) - P_{t_k}(y,\cdot)\|_{d_k} = 0,
$$
where $\nb(x)$ is the collection of open sets containing $x$, $P$ is the transition probabilities associated
to $\mathscr{P}$, and $\|P_1 - P_2\|_{d_k}$ is the norm given by
$$
\|P_1 - P_2\|_{d_k} = \inf \int_{\mathfrak{X}^2} d_k(w,z) \Pi(\d w, \d z),
$$
the infimum being taken over all positive measures on $\mathfrak{X}^2$ with marginals $P_1$ and $P_2$.
\end{defin}

The idea behind the asymptotic strong Feller condition is that ergodicity is preserved even if the stochastic forcing is restricted
to a few unstable modes, and dissipated in the others.
Using this idea, the ergodicity of the 2-D stochastic INSE with degenerate noise has been established (see \cite{HM2006}).
Some results of ergodicity for the 3-D INSEs driven by mildly degenerate noise
relying on the strong asymptotic Feller property have also established (see \cite{RX2011,RZ2009} and the references cited therein).

\subsection{The KPP equation and multiplicative noises}

The Kolmogorov-Petrovsky-Piskunov equation (KPP)
is given by
\begin{align*}
&\pd_t u = \nabla \cdot (\A(x,t) \nabla u) + h(u) + g(u) \pd_t W,\\
&u|_{t=0}=\varphi.
\end{align*}
Attention is often restricted to the case in which $g$ and $h$ both vanish at the
two points $a, b\in \mathbb{R}$, and $g, h > 0 $ on $(a,b)$.
In this way, the asymptotic size is controlled in $L^1$.

It has been shown in Chueshov-Villermot \cite{CV1996a,CV1996b,CV1998a,CV1998b,CV1998c,CV2000}
that, for the semilinear equation with $h(u) = s g(u)$,
evolution on a bounded, open domain
with zero Neumann boundary condition
is bounded in space.
Moreover, the notion of stability in probability has also been introduced in \cite{CV1998b}:

\begin{defin}[Stability in probability]
A function $u_f$ is {\it stable in probability} if, for every $\eps > 0$,
the following relation holds:
\[
\lim_{\|\varphi - u_f\|_{L^\infty} \to 0}
\mathbb{P}\big\{ \omega \in \Omega : \sup_{t \in \mathbb{R}^+\setminus \{0\}}\|u_\varphi(s,\cdot,t,\omega) - u_f\|_{L^\infty} > \eps\big\} = 0.
\]
Otherwise, $u_f$ is {\it unstable in probability}.

A function $u_f$ is {\it globally asymptotically stable in probability} if it is stable in probability and
\[
\mathbb{P}\big\{ \omega \in \Omega : \lim_{t  \to \infty}\|u_\varphi(s,\cdot,t,\omega) - u_f\|_{L^\infty} = 0\big\} = 1.
\]
\end{defin}

By considering the moments of the spatial average,
it has been shown that the constant functions $u_1 = a$ and $u_2=b$ are fixed points whose stability in probability
depends on the values of $s$.
The results of \cite{CV2000} have been refined, say in \cite{BS2005},
and the properties of the global attractor, including the computation of exact Lyapunov exponents
in a decay scenario have been derived.

As we have remarked,
the main reason that multiplicative noises complicate the analysis of stochastic PDEs
is that one fails to have much control over the spatial average,
except when additional restrictions on the noise and initial conditions are specified.
When the noise has a root, that constant is immediately a fixed point.
This cannot be avoided even when working over the non-compact domain $\mathbb{R}$ because the $L^p$ boundedness often relies on the space that is compact,
and is a difficulty we have to overcome in order to gain a deeper understanding of the asymptotic behaviors of solutions.

\subsection{Large deviation principles}

Beyond the existence and uniqueness of invariant measures,
large deviation principles touch on their specific properties.
Whilst it goes some way outside the scope of this survey
even to introduce the theory of large deviations,
which attempts to characterize the limiting behavior of a family of probability measures
(in our case, invariant measures) depending on some parameter
by using a {\em rate function},
we should be remiss to neglect mentioning it altogether;
two vintage references to the subject are \cite{DS1989,FW1998}.
More modern treatments can be found in \cite{DE1997,DZ1998,Hol2008,Var2016}
 and the references cited therein.
 Of particular interest has been the ``zero-noise'' limit of stochastic equations in which
 one looks at the stochastic equations with a small parameter $\ep$ multiplied to the noise.
 Questions of large deviation type also arise in stochastic homogenization theory.
 Each of these subjects can justify an independent survey.
 Pertaining specifically to stochastic conservation laws, the literature is, however, more sparse.
 Going some way outside the classical Freidlin-Wentzell theory,
 some results have been announced pertaining to large deviation estimates for stochastic conservation laws.
 Specifically,
 in \cite{Mar2010}, large deviation principles have been investigated and derived
 in the limit of jointly vanishing noise and viscosity by using delicate scaling arguments.
 Notably,  in \cite{BBC2017}, the bounds for the rate function have also been derived in the vanishing
 viscosity limit only,
 so that the noise is allowed in the limit, and in the multidimensional setting.
 Finally, we mention the more recent work \cite{DWZZ2018} and the references cited therein,
 large deviation principles have been derived for the first-order scalar conservation laws with
 small {\em multiplicative} noise on $\T^d$ in the zero-noise limit
 by using the Freidlin-Wentzell theory.
 Much still remains to be explored in this direction.

\bigskip
\textbf{Acknowledgements}. $\,$
The research of Gui-Qiang G. Chen was supported in part by
the UK
Engineering and Physical Sciences Research Council Award
EP/E035027/1 and
EP/L015811/1, and the Royal Society--Wolfson Research Merit Award (UK).
The research of Peter Pang was
supported in part by the UK
Engineering and Physical Sciences Research Council Award
EP/E035027/1 and
EP/L015811/1, and an Oxford Croucher Scholarship.

\bibliographystyle{plain}

\end{document}